\documentclass{amsart}
\usepackage{amsaddr,amsfonts,amsmath,amssymb,amsthm}
\usepackage{accents}
\usepackage{bm}
\usepackage{booktabs}
\usepackage[margin=1in]{geometry}
\usepackage{graphicx}
\usepackage{hyperref}
\usepackage{makecell}
\usepackage{stmaryrd}
\usepackage{tikz}
\usepackage{pbox}

\usepackage[numbers]{natbib}

\usepackage{mdframed}
\newmdenv[fontcolor=red,linecolor=red,skipabove=12pt,skipbelow=12pt,
        innertopmargin=12pt,innerbottommargin=12pt]{note}

\def\A{\mathcal{A}}

\def\Q{\mathcal{Q}}
\def\R{\mathbb{R}}
\def\T{\mathcal{T}}
\def\DG{\mathrm{DG}}
\def\BR{\mathrm{BR}}
\def\k{\kappa}
\def\khat{\widehat{\kappa}}
\def\Qh{\mathsf{Q}_h}

\def\Ahat{\widehat{A}}
\def\Atilde{\widetilde{A}}
\def\ktilde{\widetilde{\k}}

\def\dotdot{\cdot\,,\cdot}

\def\tr{\intercal}

\def\lmax{\lambda_{\max}}
\def\lmin{\lambda_{\min}}

\def\card{\mathrm{card}}

\def\eq{\approx}

\def\llb{\llbracket}
\def\rrb{\rrbracket}

\def\interior#1{\accentset{\circ}{#1}}

\newtheorem{thm}{Theorem}
\newtheorem{lem}{Lemma}
\newtheorem{cor}{Corollary}
\newtheorem{prop}{Proposition}
\newtheorem{rem}{Remark}

\newdimen\imageheight

\title[Subspace correction preconditioners for DG methods with $hp$-refinement]{Uniform subspace correction preconditioners for discontinuous Galerkin methods with $hp$-refinement}
\author[Pazner and Kolev]{Will Pazner and Tzanio Kolev}
\address{Center for Applied Scientific Computing, Lawrence Livermore National Laboratory}

\begin{document}

\begin{abstract}
   In this paper, we develop subspace correction preconditioners for discontinuous Galerkin (DG) discretizations of elliptic problems with $hp$-refinement.
   These preconditioners are based on the decomposition of the DG finite element space into a conforming subspace, and a set of small nonconforming edge spaces.
   The conforming subspace is preconditioned using a matrix-free low-order refined technique, which in this work we extend to the $hp$-refinement context using a variational restriction approach.
   The condition number of the resulting linear system is independent of the granularity of the mesh $h$, and the degree of polynomial approximation $p$.
   The method is amenable to use with meshes of any degree of irregularity and arbitrary distribution of polynomial degrees.
   Numerical examples are shown on several test cases involving adaptively and randomly refined meshes, using both the symmetric interior penalty method and the second method of Bassi and Rebay (BR2).
\end{abstract}

\maketitle

\section{Introduction}

High-order discontinuous Galerkin (DG) methods have seen significant recent interest in a wide range of application areas \cite{Bassi2000,Wang2013,Cockburn2001}.
One advantageous feature of the DG method is its flexibility, allowing for the natural handling of irregular (nonconforming) meshes and variable polynomial degrees, thus making it well-suited for $hp$-adaptive refinement \cite{Wihler2003}.
However, one challenging aspect of $hp$-refined DG methods is the iterative solution of the resulting ill-conditioned system of linear equations.
Many past works have developed and analyzed solvers and preconditioners for DG discretizations of elliptic problems, including multigrid methods \cite{Gopalakrishnan2003,Dobrev2006}, and domain decomposition methods \cite{Chung2013,Antonietti2007,Antonietti2016}, among others.
However, much of the work on solvers for DG has focused on low-order elements, and some of the results hold only in the case of $h$-refinement alone.
In the case of general $hp$-refinement, the presence of hanging vertices, highly graded meshes, and variable polynomial degrees makes the construction of efficient preconditioners more challenging.

In this work, we take a subspace correction approach, where an effective preconditioner on the conforming subspace of the DG finite element space is combined with a simple smoother, also based on a space decomposition, defined in terms of Gauss--Lobatto nodal points.
Similar ideas were developed for low-order DG methods in \cite{Dobrev2006}, and for high-order DG methods on conforming meshes in \cite{Antonietti2016,Pazner2019a}.
In the case of conforming meshes or $h$-refinement only, a simple point Jacobi smoother is sufficient to obtain uniform bounds on the preconditioned system.
However, in the case of $hp$-refinement, the conforming subspace does not provide as good of an approximation to the DG finite element space (as quantified by a Jackson-type estimate), necessitating more powerful block smoothers, which take the form of block Jacobi methods.
These block smoothers correspond to a subspace decomposition that is generated using a simple algorithm, and which depends on the irregularity of the mesh and the distribution of polynomial degrees.
The resulting preconditioned system has condition number independent of the mesh size, polynomial degree, and penalty parameter.

Past work on solvers for $hp$-DG methods include multilevel methods and domain decomposition methods.
In \cite{Antonietti2010}, a class of domain decomposition preconditioners for $hp$-DG discretizations was analyzed, where the condition number of the resulting system is bounded by $\eta p^2 H/h$, where $H$ is the coarse mesh size, and $\eta$ is the DG penalty parameter.
A multilevel approach related to the current work was developed in \cite{Brix2008}, where a \textit{minimal} conforming subspace (corresponding to a low-order $H^1$ finite element space) was used on graded (1-irregular) meshes to obtain uniform bounds on the condition number of the preconditioned system.
This work was extended in \cite{Brix2014} using a low-order preconditioner that is defined on nested dyadic grids.
This is in contrast to the present work, where a low-order preconditioner is formed using a variational restriction strategy.
In \cite{Antonietti2015}, multigrid algorithms for $hp$-DG methods were considered, and the dependence of the convergence factor on the type of multigrid cycle and number of smoothing steps was studied.

The structure of this paper is as follows.
In Section \ref{sec:discretization}, we introduce the symmetric interior penalty DG discretization for the model problem, and enumerate several useful known results.
In Section \ref{sec:decomp}, we define a novel space decomposition that is analyzed within the abstract framework of additive Schwarz methods.
This analysis depends on approximation results for an Oswald-type averaging operator that maps from the DG finite element space to its conforming subspace.
An algorithm to generate the subspace decomposition based on the mesh irregularity is described.
At the end of Section \ref{sec:decomp} we also discuss a matrix-free preconditioners for the conforming problem based on a low-order refined methodology.
Finally, in Section \ref{sec:results}, we present numerical results on both adaptively and randomly refined meshes, conforming the theoretical properties of the preconditioner.
We end with conclusions in Section \ref{sec:conclusions}.

\section{Discontinuous Galerkin discretization}
\label{sec:discretization}

Consider the model Poisson problem
\begin{equation}
   \label{eq:poisson}
   \begin{aligned}
   \Delta u &= f &&\qquad \text{ in $\Omega$,} \\
   u &= 0 &&\qquad \text{ on $\partial\Omega$,}
   \end{aligned}
\end{equation}
on the spatial domain $\Omega \subseteq \R^2$.
Let $\T$ denote a mesh of the domain $\Omega$ consisting of non-overlapping quadrilateral elements denoted $\k$.
We will consider the cases of both regular (conforming) meshes (i.e.\ those without hanging nodes) and irregular (nonconforming) meshes.
Let $\khat$ denote the reference element $\khat = [0,1]^2$.
For given $p$, let $\Q_p(\khat)$ denote the space of bivariate polynomials of degree at most $p$ in each variable.
For each element $\k \in \T$, we associate a mapping $T_\k : \khat \mapsto \k$, and a polynomial degree $p_\k$
Define the DG finite element space by
\begin{equation}
   V_h = \{ v_h \in L^2(\Omega) : v_h|_\k \circ T_\k \in \Q_{p_\k}(\khat) \text{ for all $\k \in \T$} \}.
\end{equation}
In what follows, we will assume that the mesh is quasi-uniform, and that the ratio of polynomial degrees in elements sharing a common edge remains bounded.

Let $\Gamma$ denote the mesh skeleton, given by the set of all mesh faces.
Let $e \in \Gamma$ denote an interface bordering elements $\k^-$ and $\k^+$, and let $\bm n_e^-$ (resp.\ $\bm n_e^+$) denote the unit vector normal to $e$ pointing outward from $\k^-$ (resp.\ $\k^+$).
Let $\phi \in V_h$ be given.
Then, $\phi_e^-$ (resp.\ $\phi_e^+$) is used to denote the trace of $\phi$ on $e$ from within $\k^-$ (resp.\ $\k^+$)
We define the average and jump of $\phi$ at this interface by
\[
   \{ \phi \}_e = \frac{1}{2}(\phi_e^- + \phi_e^+), \qquad
   \llb \phi \rrb_e = \phi_e^- \bm n_e^- + \phi_e^+ \bm n_e^+.
\]
Typically, the edge $e$ in question will be clear from the context, and so we will often drop the subscripts $e$.

Given these definitions, we discretize \eqref{eq:poisson} using the symmetric interior penalty DG method \cite{Arnold1982,Arnold2002}.
The finite element problem is: find $u_h \in V_h$ such that, for all test functions $v_h \in V_h$,
\begin{equation} \label{eq:fe-problem}
   \A(u_h, v_h) = \int_\Omega f v_h \, d\bm x,
\end{equation}
where the bilinear form $\A(\dotdot)$ is defined by
\begin{equation}
   \A(u_h, v_h) =
      \int_\Omega \nabla_h u_h \cdot \nabla_h v_h \, d\bm x
      - \int_\Gamma \{ \nabla_h u_h \} \cdot \llb v_h \rrb \, ds
      - \int_\Gamma \{ \nabla_h v_h \} \cdot \llb u_h \rrb \, ds
      + \int_\Gamma \sigma \llb u_h \rrb \cdot \llb v_h \rrb \, ds.
\end{equation}
Here, $\nabla_h \phi$ denotes the \textit{broken gradient} of $\phi$, defined elementwise over $\T$.
The parameter $\sigma$ must be chosen sufficiently large in order to ensure that the resulting discretization is stable \cite{Arnold2002}.
In particular, at each interface, we take $\sigma = \eta\, p_e^2/h_e$, where $p_e = \max \{ p_{\k^-}, p_{\k^+} \}$, and $h_e = \min \{ h_{\k^-}, h_{\k^+} \}$, where $h_\k$ denotes the mesh size of the element $\k$ \cite{Wihler2003}.
The parameter $\eta$ is known as the \textit{penalty parameter}, and is chosen to satisfy $\eta \geq \eta_0 > 0$ for some $\eta_0$.
An explicit formula for $\eta_0$ was given in \cite{Shahbazi2005}, where it was also noted that the conditioning of the resulting linear system degrades for larger values of $\eta$.
One of the goals of this work is to develop preconditioners for this system whose convergence is independent of the choice of $\eta$.

\begin{rem}
   There are many alternative DG discretizations for elliptic problems, including the local DG (LDG) method \cite{Cockburn1998}, the compact DG (CDG) method \cite{Peraire2008}, the method of Bassi and Rebay (BR2) \cite{Bassi2000}, among others.
   For simplicity, we focus on the symmetric interior penalty (SIPDG) method, but the preconditioners developed here are also applicable to these alternative methods.
   Numerical examples using the BR2 method are presented in Section \ref{sec:penalty}.
\end{rem}

\begin{rem}
   In what follows, we will use the notation $a \lesssim b$ to mean that $a \leq C b$, where $C$ is a constant that is independent of the mesh size, polynomial degree, and penalty parameter, but which \textit{may} depend on the level of irregularity of the mesh, and on the ratio of polynomial degrees in neighboring elements.
   We will write $a \gtrsim b$ to mean $b \lesssim a$, and $a \eq b$ to mean that both $a \lesssim b$ and $b \lesssim a$.
\end{rem}

\subsection{DG norms and estimates}

We define the mesh-dependent DG norm $\|\cdot\|_{\DG}$ by
\begin{equation}
   \| v_h \|_{\DG}^2 = \| \nabla_h v_h \|_{0}^2 + \sum_{e \in \Gamma} \| \sigma^{1/2} \llb v_h \rrb \|_{0,e}^2.
\end{equation}
The bilinear form $\A(\dotdot)$ satisfies the following continuity and coercivity bounds \cite{Antonietti2010}.
\begin{lem}
   For all $v_h, u_h \in V_h$,
   \begin{align}
      \A(u_h, v_h) &\lesssim \| u_h \|_{\DG} \| v_h \|_{\DG}, \\
      \A(u_h, u_h) &\gtrsim \| u_h \|_{\DG}^2.
   \end{align}
\end{lem}
We additionally have the following useful eigenvalue estimates \cite{Antonietti2010}.
\begin{lem} \label{lem:eigenvalue}
   For any $u_h \in V_h$,
   \[
      \| u_h \|_0^2 \lesssim \A(u_h, u_h) \lesssim \sum_{\k\in\T} \eta \frac{p_\k^4}{h_\k^2} \| u_h \|_{0,\k}^2.
   \]
\end{lem}
The following simple result is a slight refinement of Lemma \ref{lem:eigenvalue}.
\begin{lem} \label{lem:eigval-2}
For any $u_h \in V_h$,
\[
   \A(u_h, u_h) \lesssim \sum_{\k\in\T} \frac{p_\k^4}{h_\k^2} \| u_h \|_{0,\k}^2
   + \sum_{e\in\Gamma} \eta \frac{p_e^2}{h_e} \| \llb u_h \rrb \|_{0,e}^2.
\]
Furthermore, for any $w_h\in V_h$ with $\llb w_h \rrb = 0$,
\[
   \A(u_h - w_h, u_h - w_h)
   \lesssim \sum_{\k\in\T} \frac{p_\k^4}{h_\k^2} \| u_h - w_h \|_{0,\k}^2
   + \sum_{e\in\Gamma} \eta \frac{p_e^2}{h_e} \| \llb u_h \rrb \|_{0,e}^2.
\]
\end{lem}
\begin{proof}
   The first statement follows from Poincar\'e's inequality, and the second statement follows from linearity of the jump.
\end{proof}

\subsection{Gauss--Lobatto basis}

A key ingredient in the construction of preconditioners for the DG discretization will be the use of the Gauss--Lobatto nodal basis.
Given polynomial degree $p_\k$, let $\widehat{\bm\xi_i} \in \khat$, $1 \leq i \leq (p+1)^2$ denote the tensor-product Gauss--Lobatto points.
For each element $\k$, the points $\bm\xi_{\k,i}$ are defined as the image of $\widehat{\bm\xi}_i$ under the mapping $T_\k$, i.e.\ $\bm\xi_{\k,i} = T_\k(\widehat{\bm\xi}_i)$.
For each such point, we define the basis function $\phi_{\k,i}$ as the unique element of $V_h$ such that
\[
   \phi_{\k,i}(\bm\xi_{\k',i'}) = \begin{cases}
      1, & \text{if $(\k',i') = (\k,i)$,} \\
      0, & \text{otherwise.}
   \end{cases}
\]
The set of functions $\{ \phi_{\k,i} \}$ forms a basis for the space $V_h$.
For a given element $\k\in\T$, let $\mathcal{B}(\k)$ denote the set of Gauss--Lobatto nodes lying on $\partial\k$, and let $\mathcal{I}(\k)$ denote the set of nodes lying in the interior of $\k$.
It is well-known that the discrete $L^2$ norm associated with the Gauss--Lobatto quadrature nodes and weights is equivalent to the (exactly integrated) $L^2$ norm \cite{Canuto1982,Canuto1994,Canuto2006}.
As a consequence, we have the following closely related result from \cite{Antonietti2016}.
\begin{lem}[{{\cite[Lemma 3]{Antonietti2016}}}] \label{lem:quadrature}
   Take any $v_h \in V_h$.
   Let $v_{\k,i}(\bm x) = v_h(\bm\xi_{\k,i}) \phi_{\k,i}(\bm x)$.
   Then,
   \[
      \| v_h \|_0^2 \eq \sum_{\k,i} \| v_{\k,i} \|_0^2.
   \]
\end{lem}
We can easily show the following result.
\begin{cor} \label{cor:quadrature-2}
   Let $v_h \in V_h$, and write $v_h = \sum_j v_j$, such that for any node $\bm\xi_{\k,i}$, the value $v_j(\bm\xi_{\k,i})$ is nonzero for at most one function $v_j$.
   Then,
   \[
      \| v_h \|_0^2 \eq \sum_j \| v_j \|_0^2.
   \]
\end{cor}
\begin{proof}
   By Lemma \ref{lem:quadrature},
   \[
      \sum_j \| v_j \|_0^2
      \approx \sum_j \sum_{\k,i} \| v_j(\bm\xi_{\k,i}) \phi_{\k,i} \|_0^2
      = \sum_{\k,i} \| v_{\k,i} \|_0^2 \eq \| v_h \|_0^2. \qedhere
   \]
\end{proof}

We will also make use of the following trace and inverse trace inequalities from \cite{Burman2007}.
\begin{lem}[{{\cite[Lemma 3.2]{Burman2007}}}]
   Let $v_h \in V_h$ be given.
   Then,
   \begin{equation} \label{eq:trace-ineq}
      \| v_h \|_{0,\partial\k}^2 \lesssim \frac{p_\k^2}{h_\k} \| v_h \|_{0,\k}^2.
   \end{equation}
   Now, suppose that $v_h$ vanishes at all interior Gauss--Lobatto nodes of an element $\k$, i.e.\ $v_h(\bm\xi) = 0$ for all $\bm\xi \in \mathcal{I}(\k)$.
   Then,
   \begin{equation} \label{eq:inv-trace-ineq}
      \| v_h \|_{0,\k}^2 \lesssim \frac{h_\k}{p_\k^2} \| v_h \|_{0,\partial\k}^2.
   \end{equation}
\end{lem}

\section{Preconditioning and space decomposition}
\label{sec:decomp}

\subsection{Parallel subspace corrections}

The preconditioner for the DG discretization will be constructed using the framework of \textit{parallel subspace corrections} (additive Schwarz methods), which has been studied in great detail in numerous works \cite{Dryja1990,Xu1992,Griebel1995,Xu2001,Xu2002,Toselli2005}.
In this section we briefly describe the abstract framework, and enumerate some results that will be useful in what follows.

Let $W$ be a vector space with inner product $(\dotdot)$.
Let $A:W \to W$ be a symmetric positive-definite linear operator, and let $\A(\dotdot)$ denote the induced inner product, i.e.\ $\A(u, v) = (Au, v)$.
The space $W$ is decomposed into a sum of subspaces, $W = \sum_{i=0}^J W_i$.
For each $i$, define the $L^2$ projection $Q_i$ by
\[
   (Q_i w, w_i) = (w, w_i) \quad\text{for all $w_i\in W_i$},
\]
and define the elliptic projection $P_i$ by
\[
   \A(P_i w, w_i) = \A(w, w_i) \quad\text{for all $w_i\in W_i$}.
\]
Let $A_i$ denote the restriction of $A$ to $W_i$.
We note the useful identity
\begin{equation} \label{eq:pq-identity}
   A_i P_i = Q_i A.
\end{equation}
The preconditioned system $P$ is defined by
\[
   P = \sum_{i=0}^J P_i = \left( \sum_{i=0}^J A_i^{-1} Q_i \right) A.
\]
In some circumstances, the subspace $W_i$ may be sufficiently large that $A_i^{-1}$ (and hence $P_i$) is impractical to compute.
In this case, we may replace each $A_i^{-1}$ with an approximation $R_i$ to obtain the preconditioner
\[
   B = \sum_{i=0}^J R_i Q_i.
\]
The corresponding approximate projections are denoted $T_i = R_i Q_i A$, and the preconditioned operator is written $T = \sum_{i=0}^J T_i = BA$.
Our main goal will be to estimate the \textit{iterative condition number} of the preconditioned operator, $\kappa(T) = \lmax(T)/\lmin(T)$, which will determine the speed of convergence when using the conjugate gradient method.
The main tool in our analysis of the subspace correction preconditioner will be the following useful identity (cf.\ \cite{Xu1992,Toselli2005,Xu2002}).
\begin{lem} \label{lem:inf-identity}
   For any $w \in W$ we have the identity
   \begin{equation} \label{eq:Tinv-inf-identity}
      \A(T^{-1}w, w) = \inf_{\substack{w_i \in W_i\\\sum w_i = w}} \sum_{i=0}^J \A(T_i^{-1} w_i, w_i),
   \end{equation}
   and, since $P_i$ restricted to the subspace $W_i$ equals the identity operator, we have the special case for exact projections
   \begin{equation} \label{eq:Pinv-inf-identity}
      \A(P^{-1}w, w) = \inf_{\substack{w_i \in W_i\\\sum w_i = w}} \sum_{i=0}^J \A(w_i, w_i).
   \end{equation}
\end{lem}

\subsection{Conforming and boundary subspaces}

In this section we describe the subspace decomposition for the DG finite element space $V_h$.
Let $V_C$ denote the \textit{conforming subspace} of $V_h$, i.e.\ $V_C = V_h \cap H^1(\Omega)$.
Let $V_B$ (where $B$ here stands for boundary) denote the set of functions $v_b$ such that $v_b(\bm\xi_{\k,i}) = 0$ for all $\bm\xi_{k,i} \in \mathcal{B}(\k)$ for all elements $\k\in\T$.
That is to say, a function $v_b \in V_B$ vanishes at every \textit{interior} Gauss--Lobatto node of every element in the mesh.
It is clear that $V_h = V_B + V_C$.

\subsubsection{Conforming approximation}
\label{sec:conf-approx}

Let $u_h \in V_h$ be given.
We are interested in approximating $u_h$ by a conforming function $u_c \in V_C$.
We will make use of an interpolation operator $\Qh : V_h \to V_C$ that is often referred to as the Oswald operator (cf.\ \cite{Burman2007,Oswald1993}, among others).
The operator $\Qh$ is defined as follows.
Consider a nodal Gauss--Lobatto basis for the conforming space $V_C$.
In the case of $h$-refinement, the nodal points on a nonconforming edge are chosen to be the corresponding Gauss--Lobatto nodes of the coarse element.
In the case $p$-refinement, the nodal points on an element interface are chosen to be the corresponding Gauss--Lobatto nodes of the element with lower polynomial degree.
Note that in the case of $hp$-refinement, the conforming nodal points are no longer a subset of the DG nodal points.
Instead, the nodes on an $hp$-interface are chosen to be the coarse element Gauss--Lobatto nodes corresponding to the lowest polynomial degree of any element containing the given edge.
Then, any conforming function $u_c \in V_C$ is well-defined given its value at all conforming nodal points $\bm\xi$.
Let $u_h \in V_h$ be given.
We define $\Qh u_h(\bm\xi)$ to be the average value of $u_h(\bm\xi)$ over all elements $\k$ containing the node $\bm\xi$,
\[
   \Qh u_h(\bm \xi) = \frac{1}{\card \{\k\in\T : \bm\xi \in \k\}} \sum_{\k\ni\bm\xi} u_h|_\k(\bm\xi).
\]

If the mesh $\T$ is conforming, then it is possible to show the following important $hp$ approximation property of $\Qh$:
\begin{equation} \label{eq:Qh-jump}
   \| v_h - \Qh v_h \|_{0}^2
   \leq C \sum_{e\in\Gamma} \frac{h_\k}{p_\k^2} \| \llb v_h \rrb \|_{0,e}^2.
\end{equation}
This result was shown for conforming meshes and uniform polynomial degree in \cite{Burman2007}.
The case of nonconforming meshes and uniform polynomial degree was considered in \cite{Karakashian2003}.
The case of conforming meshes and variable polynomial degree was considered in \cite{Houston2007}.
Additionally, similar results were shown for one-irregular meshes with variable polynomial degree using an auxiliary mesh technique in \cite{Houston2007a,Zhu2010,Zhu2011}.

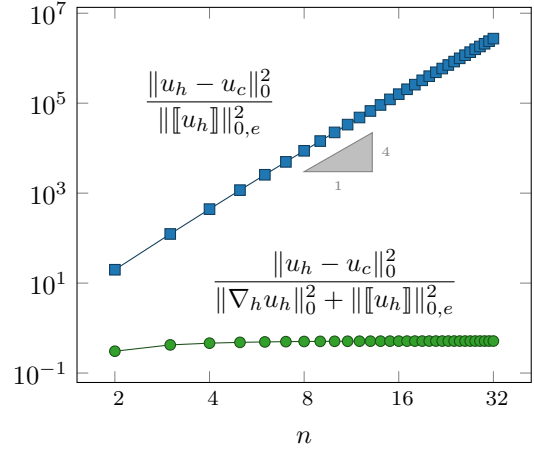
\begin{figure}
   \centering
   \pbox{0.4\linewidth}{
      \begin{tikzpicture}
         \draw (0,0) rectangle ++(6,3);
         \draw (3,0) -- ++(0,3);
         \draw (3,1) -- ++(3,0);
         \draw (3,2) -- ++(3,0);
         \node at (1.3,3.4) {$e$};
         \draw [->] (1.5,3.4) to[out=0,in=180] (3,1.5);
         \node at (1.5,1.5) {$p=2$};
         \node at (1.5,0.7) {$\k_0$};
         \node at (4.5,0.5) {$p=1 \qquad \k_1$};
         \node at (4.5,1.5) {$\vdots$};
         \node at (4.5,2.5) {$p=1 \qquad \k_n$};
         \draw[fill] (3,0) circle (2.5pt);
         \draw[fill] (3,3) circle (2.5pt);
      \end{tikzpicture}
   }
   \hspace{1cm}
   \pbox{0.59\linewidth}{
      \includegraphics{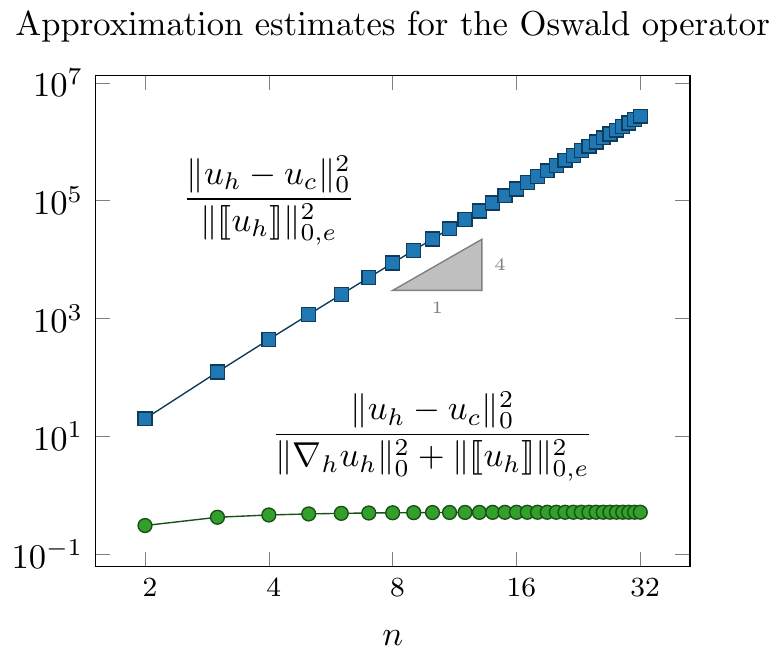}
   }

   \caption{
      Left: example of problematic case for estimates of type \eqref{eq:Qh-jump}.
      A large high-order element ($p=2$) borders $n$ low-order ($p=1$) elements.
      The value of any conforming function at the nonconforming interface $e$ is determined by its values at the two indicated nodes.
      Right: numerically computed best value of the constants in estimates \eqref{eq:Qh-jump} and \eqref{eq:Qh-jump-grad} for this case.
   }
   \label{fig:mesh-example}
\end{figure}

However, in the case of general nonconforming meshes with variable polynomial degrees, estimates of the form \eqref{eq:Qh-jump} are not satisfactory and so one of the contributions of this paper is to propose a more precise estimate, see \eqref{eq:Qh-jump-grad} below, that will be critical in the preconditioner analysis.

To see why \eqref{eq:Qh-jump} is not sufficient in the general $hp$ case, consider the simple mesh shown in Figure \ref{fig:mesh-example}.
Let $f$ be a given biquadratic function, and define $u_h \in V_h$ by pointwise interpolation of $f$ at the Gauss--Lobatto nodes.
Since the left element has degree $p=2$, we have $u_h|_{\k_0} = f$.
On the remaining elements, $u_h$ is given by piecewise bilinear interpolation of this function, and so the pointwise error will scale as $1/n^2$, where $n$ is the number of refined elements on the right.
As a result, we see that
\[
   \| \llb u_h \rrb \|_{0,e}^2 \lesssim 1/n^4.
\]

Let $u_c = \Qh u_h \in V_c$.
Notice that if we refine the elements on the right (i.e.\ increase $n$), the error on the interface $\| u_h - u_c \|_{0,e}$ remains unchanged because $u_c$ is always determined by its values at the two black nodes indicated in the diagram.
We attempt to estimate the constant in the inequality \eqref{eq:Qh-jump}.
We have
\begin{align*}
   C &\geq \frac{p_e^2}{h_e} \frac{\| u_h - u_c \|_0^2}{\| \llb u_h \rrb \|_{0,e}^2} \\
     &\gtrsim \frac{p_e^2}{h_e} n^4 \| u_h - u_c \|_0^2 \\
     &\gtrsim n^4 \| u_h - u_c \|_{0,e}^2
     \eq n^4,
\end{align*}
where the second-to-last step follows from the inverse trace inequality.
We see that the constant in this inequality degrades very quickly with the number of nonconforming refinements in the $hp$ case.
This estimate is verified numerically by computing the minimum value of $C$ such that that inequality \eqref{eq:Qh-jump} holds for this particular configuration.
The values are shown in Figure \ref{fig:mesh-example}, illustrating the degradation of the constant $C$ with increasing refinements.

To address this issue, we bound the difference $\| u_c - u_c \|_0^2$ by both $\| \nabla u_h \|_0^2$ and $\sum_{e\in\Gamma} \frac{h_e}{p_e^2} \| \llb u_h \rrb \|_{0,e}^2$.
The numerical computations in Figure \ref{fig:mesh-example} indicate that the resulting modified estimates remain constant with increasing refinements.
This is confirmed by the following lemma, which is closely related to the Jackson-type estimates of \cite{Brix2008}.
\begin{lem} \label{lem:l2-jump-h1}
   Let $u_h \in V_h$, and let $u_c = \Qh u_h \in V_C$.
   Then,
   \begin{equation} \label{eq:Qh-jump-grad}
      \| u_h - u_c \|_{0}^2
      \lesssim \sum_\k \frac{h_\k^2}{p_\k^4} \| \nabla u_h \|_{0,\k}^2
      + \sum_{e\in\Gamma} \frac{h_e}{p_e^2} \| \llb u_h \rrb \|_{0,e}^2.
   \end{equation}
\end{lem}
\begin{proof}
   Let $u_h \in V_h$ be given.
   We construct a function $\tilde{u} \in V_h$ element-by-element as follows.
   If a nodal point $\bm\xi_{\k,i}$ does not lie on a nonconforming edge (i.e. it is an interior node, $\bm\xi_{\k,i}\in\mathcal{I}(\k)$, or it lies on a conforming edge), then we set $\tilde{u}(\bm\xi_{\k,i}) = u_h(\bm\xi_{\k,i})$.
   It remains to define $\tilde{u}$ on the nonconforming edges of the mesh.
   Let $e$ denote a nonconforming edge, bordering coarse element $\k_0$, and fine elements $\k_1, \ldots, \k_n$.
   Let $p_e = \min_{i} p_{\k_i}$, and let $\bm\xi_j$ denote the $p_e+1$ Gauss-Lobatto nodes on $e$ (these are the \textit{conforming} nodal points defined at the beginning of Section \ref{sec:conf-approx}).
   Then, let $\tilde{u}_{\k_0,e}$ is chosen to be the degree-$p_e$ polynomial that interpolates $u_h|_{\k_0}$ at the points $\bm\xi_j$.
   For $i > 0$, if $p_{\k_i} = p_e$, then $\tilde{u}|_{\k_i,e} = u_h|_{\k_i,e}$.
   However, if $p_{\k_i} > p_e$, we choose $\tilde{u}|_{\k_i,e}$ to be a degree-$p_e$ interpolant as follows.
   Note that fewer than $p_e+1$ of the points $\bm\xi_j$ lie within $e\cap\k_i$.
   We select interpolation points $\bm\zeta_{i,j}$ that consist of those nodal points $\bm\xi_j$ that lie within $e \cap \k_i$, supplemented with additional Gauss-Lobatto points to obtain $p_e+1$ distinct points lying within $e \cap \k_i$.
   Then, $\tilde{u}_{\k_i,e}$ is chosen to interpolate $u_h|_{\k_i}$ at the $p_e+1$ points $\bm\zeta_{i,j}$.

   By this definition, we have $\tilde{u}|_\k(\bm\xi) = u|_\k(\bm\xi)$ for all conforming nodal points $\bm\xi$.
   Consequently, $\Qh \tilde{u} = \Qh u_h = u_c$.
   However, for any nonconforming interface, $\tilde{u}|_{e}$ has degree no higher than $p_e$.
   Therefore, we can apply the arguments of \cite[Proposition 5.2]{Houston2007} and \cite[Theorem 2.3]{Karakashian2003} to show that
   \[
      \| \tilde{u} - u_c \|_0^2
      \lesssim \sum_{e\in\Gamma} \frac{h_e}{p_e^2} \| \llb \tilde{u} \rrb \|_{0,e}^2.
   \]
   Additionally, we have
   \[
      \| u_h - u_c \|_0^2
      \lesssim \| u_h - \tilde{u} \|_0^2 + \| \tilde{u} - u_c \|_0^2
      \lesssim \| u_h - \tilde{u} \|_0^2 + \sum_{e\in\Gamma} \frac{h_e}{p_e^2} \| \llb \tilde{u} \rrb \|_{0,e}^2.
   \]
   By accuracy of the polynomial interpolant,
   \[
      \| u_h - \tilde{u} \|_0^2
      = \sum_\k \| u_h - \tilde{u} \|_{0,\k}^2
      \lesssim \sum_\k \frac{h_\k^2}{p_{\k'}^4} \| \nabla u_h \|_{0,\k}^2,
   \]
   where $p_{\k'}$ is the minimum polynomial degree of all elements $\k'$ neighboring $\k$ (note that we make that assumption that the ratio of polynomial degrees on neighboring elements is bounded, and so $p_{\k'} \eq p_\k$).

   It remains to estimate the term $\sum_{e\in\Gamma} \frac{h_e}{p_e^2} \| \llb \tilde{u} \rrb \|_{0,e}^2$.
   On a given edge $e$, we have
   \begin{align*}
      \| \llb \tilde{u} \rrb \|_{0,e}
      &= \| \llb \tilde{u} \rrb - \llb u_h \rrb + \llb u_h \rrb \|_{0,e}
      = \| \tilde{u}^- - \tilde{u}^+ - u_h^- + u_h^+ + u_h^- - u_h^+ \|_{0,e} \\
      &\leq \| \tilde{u}^- - u_h^- \|_{0,e} + \| \tilde{u}^- - u_h^- \|_{0,e} + \| \llb u_h \rrb \|_{0,e}.
   \end{align*}
   Again using accuracy of the interpolant and the trace inequality \eqref{eq:trace-ineq}, we obtain
   \[
      \frac{h_e}{p_e^2} \| \tilde{u}^\pm - u_h^\pm \|_{0,e}^2
      \lesssim \| \tilde{u} - u_h \|_{0,k^\pm}^2
      \lesssim \frac{h_e^2}{p_e^4} \| \nabla u_h \|,
   \]
   and so
   \[
      \sum_{e\in\Gamma} \frac{h_e}{p_e^2} \| \llb \tilde{u} \rrb \|_{0,e}^2
      \lesssim \sum_\k \frac{h_\k^2}{p_\k^4} \| \nabla u_h \|_{0,\k}^2
      + \sum_{e\in\Gamma} \frac{h_e}{p_e^2} \| \llb u_h \rrb \|_{0,e}^2,
   \]
   from which the conclusion \eqref{eq:Qh-jump-grad} follows.
\end{proof}

Now, we consider some special cases in which the stronger estimates of the form \eqref{eq:Qh-jump} hold.
The following lemmas are generalizations of the $hp$ estimates from \cite{Burman2007} to the case of nonconforming meshes.
Let $\nu$ denote a vertex of the mesh $\T$, and let $K_\nu$ denote the set of all elements $\k\in\T$ containing $\nu$ as a vertex.
A vertex is called \textit{hanging} if it is contained in an element of which it is not a vertex (i.e.\ $\nu \in \k$ where $\k \notin K_\nu$).
A vertex is called \textit{regular} if it not hanging.
That is, a vertex $\nu$ is regular if for each $\k$ such that $\nu\in\k$, we have $\k\in K_\nu$.
For a given vertex $\nu$, for each element $\k\in K_\nu$, there exists a Gauss--Lobatto node $\bm\xi_{\k,i}$ that is coincident with $\nu$.
Let $\Xi_\nu$ denote the set of all such Gauss--Lobatto nodes coincident with $\nu$.

\begin{lem} \label{lem:regular-vertex}
   Let $u_h \in V_h$ be given, and let $v_b = u_h - \Qh u_h$.
   Let $\nu$ be a regular vertex.
   Then,
   \[
      \sum_\k \sum_{\bm\xi_{\k,i} \in \Xi_\nu} \| v_b(\bm\xi_{\k,i}) \phi_{\k,i} \|_{0,\k}^2
      \lesssim \sum_{e \ni \nu} \frac{h_e}{p_e^2} \| \llb u_h \rrb \|_{0,e}^2.
   \]
\end{lem}
\begin{proof}
   Applying the inverse trace inequality \eqref{eq:inv-trace-ineq} twice, we obtain
   \[
      \| v_b(\bm\xi_{\k,i}) \phi_{\k,i} \|_{0,\k}^2
      \lesssim \frac{h_\k^2}{p_\k^4} \| v_b(\bm\xi_{\k,i}) \phi_{\k,i} \|_{0,\nu}^2.
   \]
   Recall that $\Qh u_h$ is defined by
   \[
      \Qh u_h(\nu) = \frac{1}{\card(K_\nu)} \sum_{\k' \in K_\nu} u_h|_{\k'}(\nu),
   \]
   and so
   \[
      v_b|_\k(\nu)
      = u_h|_\k(\nu) - \frac{1}{\card(K_\nu)} \sum_{\k' \in K_\nu} u_h|_{\k'}(\nu),
   \]
   which can be written as, for appropriate choice of coefficients $\alpha_e$,
   \[
      v_b|_\k(\nu) = \sum_{e\ni\nu} \alpha_e \llb u_h \rrb_e.
   \]
   By the trace inequality \eqref{eq:trace-ineq}, we have
   \[
      \| \llb u_h \rrb_e \|_{0,\nu}^2 \lesssim \frac{p_e^2}{h_e} \| \llb u_h \rrb \|_{0,e}^2,
   \]
   and the desired result follows.
\end{proof}

\begin{lem} \label{lem:conf-edge}
   Let $e$ be a conforming edge (i.e.\ no hanging vertex lies in the interior of $e$, denoted $\interior{e}$), and suppose that $p_{\k^-} = p_{\k^+}$, where $\k^-$ and $\k^+$ are the two elements containing $e$.
   Let $u_h \in V_h$ be given, and let $v_b = u_h - \Qh u_h$.
   Then,
   \[
      \sum_{\k\in\{\k^-,\k^+\}} \sum_{\bm\xi_{\k,i} \in \interior{e}} \| v_b(\bm\xi_{\k,i}) \phi_{\k,i} \|_{0,\k}^2
      \lesssim \frac{h_e}{p_e^2} \| \llb u_h \rrb \|_{0,e}^2.
   \]
\end{lem}
\begin{proof}
   Let $\bm\xi_i \in \interior{e}$ denote a Gauss--Lobatto node lying on the interior of $e$, and let $\phi^\pm$ denote the corresponding basis functions.
   Then, $\Qh u_h(\bm\xi_i) = \frac{1}{2} (u_h^-(\bm\xi_i) + u_h^+(\bm\xi_i))$ and so $v_b^\pm(\bm\xi_i) = \pm\frac{1}{2} \llb u_h \rrb (\bm\xi_i)$.
   By Lemma \ref{lem:quadrature} and the inverse trace inequality,
   \begin{align*}
      \sum_{\k\in\{\k^-,\k^+\}} \sum_{\bm\xi_{\k,i} \in \interior{e}} \| v_b(\bm\xi_{\k,i}) \phi_{\k,i} \|_{0,\k}^2
      &\eq \sum_{\k\in\{\k^-,\k^+\}} \sum_{\bm\xi_{\k,i} \in \interior{e}} \| \llb u_h \rrb (\bm\xi_{\k,i}) \phi_{\k,i} \|_{0,\k}^2 \\
      &\eq \sum_{\k\in\{\k^-,\k^+\}} \sum_{\bm\xi_{\k,i} \in \interior{e}} \frac{h_e}{p_e^2} \| \llb u_h \rrb (\bm\xi_{\k,i}) \phi_{\k,i} \|_{0,e}^2
      \eq \frac{h_e}{p_e^2} \| \llb u_h \rrb \|_{0,e}^2. \qedhere
   \end{align*}
\end{proof}

We will call a vertex $\nu$ of the mesh $\T$ an \textit{$hp$-vertex} if $\nu$ is a hanging vertex, and the elements containing $\nu$ do not all have the same polynomial degree.
In the case that an edge does not contain any $hp$-vertices, we can apply the result of Karakashian and Pascal \cite{Karakashian2003} to obtain the following result.
\begin{lem} \label{lem:no-hp}
   Let $e$ denote an edge that does not contain any $hp$-vertices.
   Let $u_h \in V_h$ be given, and let $v_b = u_h - \Qh u_h$.
   Then,
   \[
      \sum_{\k \ni e} \| v_b \|_{0,k}^2
      \lesssim  \frac{h_e}{p_e^2} \| \llb u_h \rrb \|_{0,e}^2.
   \]
\end{lem}

\begin{rem}\label{rem:cases}
   In practice, we observe that certain nonconforming interfaces not included in the above three cases also satisfy estimate \eqref{eq:Qh-jump}.
   In particular, this is observed for interfaces for which the coarse element has polynomial degree no higher than any of the fine elements containing the given face.
   This is corroborated by numerical examples shown in Section \ref{sec:results}, however it is not implied as a consequence of the above estimates.
\end{rem}

\subsection{Preconditioning}

Let $P_C$ denote elliptic projection onto $V_C$, and likewise let $P_B$ denote elliptic projection onto $V_B$.
We will show that the additive Schwarz system
\[
   P = P_B + P_C
\]
is uniformly well-conditioned with respect to $h$, $p$, and $\eta$.
First, we show that the decomposition $V_h = V_B + V_C$ is stable.
\begin{lem} \label{lem:stable-decomp}
   Let $u_h \in V_h$, and write $u_h = u_b + u_c$, where $u_c = \Qh u_h$.
   Then,
   \[
      \A(u_b, u_b) \lesssim \A(u_h, u_h)
   \]
   and
   \[
      \A(u_c, u_c) \lesssim \A(u_h, u_h).
   \]
\end{lem}
\begin{proof}
   By the estimates in Lemma \ref{lem:l2-jump-h1} and Lemma \ref{lem:eigval-2}, we have
   \begin{align*}
      \A(u_h - u_c, u_h - u_c)
      &\lesssim \sum_{\k\in\T} \frac{p_\k^4}{h_\k^2} \| u_h - \Qh u_h \|_{0,\k}^2
      + \sum_{e\in\Gamma} \eta \frac{p_\k^2}{h_\k} \| \llb u_h \rrb \|_{0,e}^2 \\
      &\lesssim \| \nabla_h u_h \|_{0,\k}^2 + \sum_{e\in\Gamma} \eta \frac{h}{p^2} \| \llb u_h \rrb \|_{0,e}^2 \\
      &\eq \| u_h \|_{\DG}^2 \lesssim \A(u_h, u_h),
   \end{align*}
   proving the first assertion.
   The second assertion follows by writing $u_c = (u_h - u_c) + u_c$ and using the triangle inequality.
\end{proof}

\begin{thm}
   The iterative condition number of $P$ satisfies
   \[
      \kappa(P) = \lmax(P)/\lmin(P) \eq 1.
   \]
\end{thm}
\begin{proof}
   Since both $P_B$ and $P_C$ are projections, we have $\lmax(P) \leq 2$.
   To establish the bound for the minimum eigenvalue of $P$, we use identity \eqref{eq:Pinv-inf-identity} from Lemma \ref{lem:inf-identity},
   \[
      \A(P^{-1} v_h, v_h) = \inf_{v_h = v_b + v_c} \left( \A(v_b, v_b) + \A(v_c, v_c) \right).
   \]
    Lemma \ref{lem:stable-decomp} shows that $V_B + V_C$ is a stable decomposition of $V_h$, i.e.\ for all $v_h \in V_h$, there exist $v_b \in V_B$, $v_c \in V_C$ such that $v_h = v_b + v_c$ and
   \[
      \A(v_b, v_b) + \A(v_c, v_c) \lesssim \A(v_h, v_h).
   \]
   Consequently, we have
   \[
      \A(P^{-1} v_h, v_h) \lesssim \A(v_h, v_h),
   \]
   and the desired result follows.
\end{proof}

The spaces $V_B$ and $V_C$ are themselves large, and the computation of $P_B$ and $P_C$ requires the inversion of the bilinear forms $\A_B$ and $\A_C$, respectively.
This cost is clearly prohibitive, and so we seek to replace $P_B$ and $P_C$ with approximations $T_B$ and $T_C$.
On the conforming space $V_C$, note that $\A_C$ corresponds to a standard $H^1$-conforming discretization, and so we may replace $\A_C^{-1}$ with a uniform preconditioner for the conforming problem.
In this work, we use a low-order refined matrix-free preconditioner, which is described in greater detail in Section \ref{sec:lor}.
In principle, any uniform preconditioner for $\A_C$ may be used, and in the remainder of this section we will assume that $\A( T_C^{-1} u_c, u_c ) \eq \A(u_c, u_c)$ for all $u_c \in V_C$.

The projection $P_B$ onto the boundary space $V_B$ is approximated using a further space decomposition.
We decompose the space $V_B$ as the sum of yet-to-be-defined subspaces
\begin{equation} \label{eq:VB-subspaces}
   V_B
   = V_E + V_J,
   \qquad V_E = \sum_{e} V_e,
   \qquad V_J = \sum_j V_j,
\end{equation}
and define the corresponding approximate projection $T_B$ by
\begin{equation} \label{eq:TB}
   T_B = T_E + T_J,
   \qquad T_E = \sum_e P_e,
   \qquad T_J = \sum_j P_j.
\end{equation}

We begin by defining the space $V_J$.
The subscript $J$ is used to indicate that the approximate projection $T_J$ onto $V_J$ will be a simple point Jacobi method.
For this reason, it is advantageous to choose $V_J$ to be as large as possible while still obtaining a stable decomposition.
Let $V_J$ consist of any degree of freedom coincident with a regular mesh vertex, lying on the interior of a conforming edge with uniform polynomial degrees, or lying on an edge which contains no $hp$-vertices (see the definitions in Section \ref{sec:conf-approx}).
In other words, $V_J$ is designed to consist of those freedom for which the stronger approximation estimate \eqref{eq:Qh-jump} holds for the Oswald operator $\Qh$.
As noted in Remark \ref{rem:cases}, the conditions above are not exhaustive, and in practice $V_J$ can be chosen to also include those degrees of freedom lying on edges for which the coarse element has polynomial degree no higher than any of the neighboring fine elements.
For the $j$th degree of freedom in the space $V_J$, let $V_j$ denote the corresponding one-dimensional subspace, so that $V_J$ can be written as the direct sum $V_j = \sum_j V_j$.

\begin{figure}
   \centering
   \includegraphics{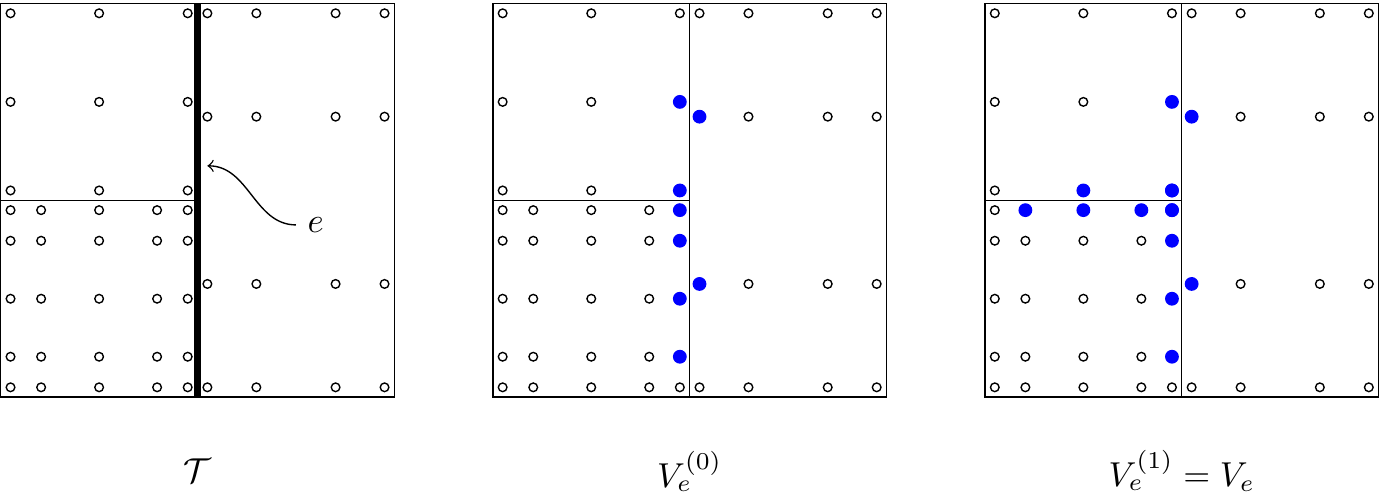}
   \caption{Example of the creation of the subspaces $V_e$ on a non-conforming mesh with variable polynomial degrees.
   The subspaces are spanned by the basis functions corresponding to the indicated nodes.
   Note that nodes that are coincident with regular mesh vertices are omitted.}
   \label{fig:V_e}
\end{figure}

Now, we define the spaces $V_e$.
For every edge $e\in\Gamma$ define a (potentially empty) subspace $V_e$ according to the following process:
\begin{itemize}
   \item $V_e^{(0)}$ is given as the span of all basis functions $\phi_{\k,i} \notin V_J$ whose corresponding Gauss--Lobatto node $\bm\xi_{\k,i}$ lies on the edge $e$.
   \item $V_e^{(i+1)}$ is defined as the span of all basis functions $\phi_{\k,i} \in V_B, \phi_{\k,i} \notin V_J$ such that $\llb \phi_{\k,i} \rrb \cdot \llb v \rrb \not\equiv 0$ for some $v \in V_e^{(i)}$.
   \item The iteration terminates when $V_e^{(i+1)} = V_e^{(i)}$, at which point we set $V_e = V_e^{(i)}$.
\end{itemize}
An example of this process is illustrated in Figure \ref{fig:V_e}.
For any pair of spaces $V_e$ and $V_{e'}$, either $V_e = V_{e'}$ or $V_e \cap V_{e'} = \{ 0 \}$, and in such a case we may simply omit one of the two spaces.
So, without loss of generality, we may assume that the spaces $V_e$ are disjoint.
Furthermore, $V_e \cap V_J = \{ 0 \}$ for all $e$ by construction, and so the decomposition $V_B = \sum_e V_e + \sum_j V_j$ is in fact a direct sum.

\begin{rem}
   The process described above is guaranteed to terminate, however the spaces $V_e$ could, in principle, be quite large.
   In practice, this occurs only in pathological cases, and in realistic cases these spaces remain relatively small.
   Furthermore, we observe in the numerical results in Section \ref{sec:random-refinement} that for 1-irregular meshes, the dimension of the spaces $V_e$ remain bounded with increasing refinements.
\end{rem}

As a consequence of this choice of subspaces, we have the following result concerning the stability of the decomposition.

\begin{lem} \label{lem:j-e-stable-decomp}
   Let $u_h \in V_h$ be given, and let $v_b = u_h - \Qh u_h \in V_B$.
   Let $v_b = \sum_e v_e + \sum_j v_j$ be the unique representation of $v_b$ in the subspace decomposition \eqref{eq:VB-subspaces}.
   Then,
   \begin{equation} \label{eq:j-l2-jump}
      \sum_\k \frac{p_\k^4}{h_\k^2} \sum_j \| v_j \|_{0,\k}^2
      \lesssim \sum_\k \sum_{e'\in\partial\k} \frac{p_{\k}^2}{h_{\k}} \| \llb u_h \rrb \|_{0,e'}^2,
   \end{equation}
   and
   \[
      \sum_{e'\in\Gamma} \frac{p_{e'}^2}{h_{e'}} \left(
         \sum_e \| \llb v_e \rrb \|_{0,e'}^2
         + \sum_j \| \llb v_j \rrb \|_{0,e'}^2
      \right)
      \lesssim \sum_{e'\in\Gamma} \frac{p_{e'}^2}{h_{e'}} \| \llb u_h \rrb \|_{0,e'}^2.
   \]
\end{lem}
\begin{proof}
   The first assertion follows from the definition of the spaces $V_j$ and by Lemmas \ref{lem:regular-vertex}, \ref{lem:conf-edge}, and \ref{lem:no-hp}.

   We now prove the second assertion.
   By the trace inequality \eqref{eq:trace-ineq},
   \begin{equation} \label{eq:vj-tr-1}
      \sum_{e'\in\partial\k} \frac{p_{\k}^2}{h_{\k}} \sum_j \| v_j \|_{0,e'}^2
      \lesssim \sum_j \frac{p_\k^4}{h_\k^2} \| v_j \|_{0,\k}^2.
   \end{equation}
   Using \eqref{eq:j-l2-jump},
   \begin{equation} \label{eq:vj-tr-2}
      \sum_j \sum_\k \frac{p_\k^4}{h_\k^2} \| v_j \|_{0,\k}^2
      \lesssim \sum_{\k} \sum_{e'\in\partial\k} \frac{p_{\k}^2}{h_{\k}} \| \llb u_h \rrb \|_{0,e'}^2.
   \end{equation}
   Note that $\| v_j \|_{0,e'}^2 = \| \llb v_j \rrb \|_{0,e'}^2$, and so combining \eqref{eq:vj-tr-1} and \eqref{eq:vj-tr-2} gives
   \[
      \sum_{e'\in\partial\k} \frac{p_{\k}^2}{h_{\k}} \sum_j \| \llb v_j \rrb \|_{0,e'}^2
      \lesssim \sum_{\k} \sum_{e'\in\partial\k} \frac{p_{\k}^2}{h_{\k}} \| \llb u_h \rrb \|_{0,e'}^2.
   \]
   Let $v_J = \sum_j v_j$ and $v_E = \sum_e v_e$.
   Note that as a consequence of Corollary \ref{cor:quadrature-2}, we have
   \[
      \sum_{e'\in\Gamma} \frac{p_{e'}^2}{h_{e'}} \| \llb v_J \rrb \|_{0,e'}^2
      \lesssim \sum_{e'\in\Gamma} \frac{p_{e'}^2}{h_{e'}} \| \llb u_h \rrb \|_{0,e'}^2
   \]
   By the triangle inequality, writing $v_E = v_b - v_J$ and noting that $\llb v_b \rrb = \llb u_h \rrb$,
   \[
      \sum_{e'\in\Gamma} \frac{p_{e'}^2}{h_{e'}} \| \llb v_E \rrb \|_{0,e'}^2
      \lesssim \sum_{e'\in\Gamma} \frac{p_{e'}^2}{h_{e'}} \left( \| \llb v_b \rrb \|_{0,e'}^2 + \| \llb v_J \rrb \|_{0,e'}^2 \right)
      \lesssim \sum_{e'\in\Gamma} \frac{p_{e'}^2}{h_{e'}} \| \llb u_h \rrb \|_{0,e'}^2.
   \]
   Furthermore, by definition of the subspaces $V_e$, we have, for $e_1 \neq e_2$, $\llb v_{e_1} \rrb \cdot \llb v_{e_2} \rrb = 0$.
   Therefore,
   \[
      \sum_{e'} \frac{p_{e'}^2}{h_{e'}} \sum_e \| \llb v_e \rrb \|_{0,e'}^2 = \sum_{e'} \frac{p_{e'}^2}{h_{e'}} \| \llb v_E \rrb \|_{0,e'}^2,
   \]
   and the second statement follows.
\end{proof}

We also have the following lower bounds on $T_E^{-1}$ and $T_J^{-1}$.
\begin{lem} \label{lem:TE-TJ-bounds}
   \begin{equation} \label{eq:TE-inv-lower-bound}
      \A(T_E^{-1} v_E, v_E) \gtrsim \A(v_E, v_E) \quad\text{for all $v_E \in V_E$},
   \end{equation}
   and
   \begin{equation} \label{eq:TJ-inv-lower-bound}
      \A(T_J^{-1} v_J, v_J) \gtrsim \A(v_J, v_J) \quad\text{for all $v_J \in V_J$}.
   \end{equation}
\end{lem}
\begin{proof}
   To prove \eqref{eq:TE-inv-lower-bound}, we use the finite overlap property of the spaces $V_e$.
   For each subspace $V_e$, define the subdomain $\Omega_e \subseteq \Omega$ as the union of all elements $\k\in\T$ such that $v_e|_{\k} \not\equiv 0$ for some $v_e \in V_e$.
   Each element $\k$ is contained in a number of subdomains $\Omega_e$ bounded by the number of edges of $\k$.

   We then have
   \[
      \A(v_E, v_E)
      \eq \| v_E \|_{\DG}^2
      = \sum_{\k\in\T} \| \nabla v_E \|_{0,\k}^2 + \sum_{e'\in\Gamma} \| \sigma^{1/2} \llb v_E \rrb \|_{0,e'}^2.
   \]
   Note that $\| \nabla v_e \|_{0,\k}^2 = 0$ if $\k \notin \Omega_e$, and so by the finite overlap property,
   \[
      \| \nabla v_E \|_{0,\k}^2
      = \Big\| \sum_e \nabla v_e \Big\|_{0,\k}^2
      \lesssim \sum_e \| \nabla v_e \|_{0,\k}^2,
   \]
   and so
   \[
      \sum_{\k\in\T} \| \nabla v_E \|_{0,\k}^2 \lesssim \sum_{\k\in\T} \sum_e \| \nabla v_e \|_{0,\k}^2.
   \]
   Additionally, since $\llb v_{e_1} \rrb \cdot \llb v_{e_2} \rrb = 0$ for $e_1 \neq e_2$, we have
   \[
      \sum_{e'\in\Gamma} \| \sigma^{1/2} \llb v_E \rrb \|_{0,e'}^2
      = \sum_{e'\in\Gamma} \sum_e \| \sigma^{1/2} \llb v_e \rrb \|_{0,e'}^2.
   \]
   We conclude that
   \[
      \A(v_E, v_E)
      \lesssim \sum_e \left(
         \sum_{\k\in\T} \| \nabla v_e \|_{0,\k}^2
         + \sum_{e'\in\Gamma} \|\sigma^{1/2} \llb v_e \rrb \|_{0,e'}^2
      \right)
      = \sum_e \| v_e \|_{\DG}^2
      \lesssim \sum_e \A(v_e, v_e),
   \]
   and so, by Lemma \ref{lem:inf-identity},
   \[
      \A(T_E^{-1} v_E, v_E) = \sum_e \A(v_e, v_e) \gtrsim \A(v_E, v_E).
   \]

   To prove \eqref{eq:TJ-inv-lower-bound}, we use an argument from \cite{Antonietti2016}.
   Write $v_J = \sum_j v_j$.
   From the eigenvalue estimate (Lemma \ref{lem:eigenvalue}), we have
   \[
      \A(v_J, v_J)
      \lesssim \sum_{\k\in\T} \eta \frac{p_\k^4}{h_\k^2} \| v_J \|_{0,\k}^2
      \eq \sum_{\k\in\T} \eta \frac{p_\k^4}{h_\k^2} \sum_j \| v_j \|_{0,\k}^2.
   \]
   Using the inverse trace inequality and noting that $\| v_j \|_{0,e}^2 = \| \llb v_j \rrb \|_{0,e}^2$, we see
   \[
      \sum_{\k\in\T} \eta \frac{p_\k^4}{h_\k^2} \sum_j \| v_j \|_{0,\k}^2
      \lesssim \sum_{\k\in\T} \eta \frac{p_\k^2}{h_\k} \sum_j \| v_j \|_{0,\partial\k}^2
      = \sum_{\k\in\T} \eta \frac{p_\k^2}{h_\k} \sum_j \| \llb v_j \rrb \|_{0,\partial\k}^2
      \lesssim \sum_j \| v_j \|_{\DG}^2.
   \]
   Then, by Lemma \ref{lem:inf-identity},
   \[
      \A(T_J^{-1} v_J, v_J)
      = \sum_j \A(v_{j}, v_{j})
      \gtrsim \sum_j \| v_{j} \|_{\DG}^2
      \gtrsim \A(v_J, v_J),
   \]
   proving \eqref{eq:TJ-inv-lower-bound}.
\end{proof}
We are now ready to show that $T_B$ satisfies the following bounds.
\begin{lem} \label{lem:TB-inv-bounds}
   \begin{equation} \label{eq:TB-lower-bound}
      \A(T_B^{-1} v_b, v_b) \gtrsim \A(v_b, v_b) \quad\text{for all $v_b \in V_B$},
   \end{equation}
   and
   \begin{equation} \label{eq:TB-upper-bound}
      \A(T_B^{-1} (v_h - \Qh v_h), v_h - \Qh v_h) \lesssim \A(v_h - \Qh v_h, v_h - \Qh v_h) \quad\text{for all $v_h \in V_h$}.
   \end{equation}
\end{lem}
\begin{proof}
   We begin by proving the lower bound \eqref{eq:TB-lower-bound}.
    By Lemma \eqref{lem:inf-identity}, Lemma \eqref{lem:TE-TJ-bounds}, and the triangle inequality, we see, for $v_b = v_E + v_J$,
   \begin{align*}
      \A(T_B^{-1} v_b, v_b)
      &\gtrsim \A(T_E^{-1} v_E, v_E) + \A(T_J^{-1} v_J, v_J) \\
      &\gtrsim \A(v_E, v_E) + \A(v_J, v_J) \\
      &\geq \A(v_b, v_b),
   \end{align*}
   proving \eqref{eq:TB-lower-bound}.

   Now we turn to the upper bound \eqref{eq:TB-upper-bound}.
   This is equivalent to showing that $\sum_e V_e + \sum_{j} V_{j}$ is a stable decomposition of $\mathrm{Ran}(I - \Qh)$.
   Let $v_h$ be given, and let $v_b = v_h - \Qh v_h$.
   By Lemma \ref{lem:eigval-2},
   \begin{align} \label{eq:ATB-inv-upper-bound}
      \A(T_B^{-1} v_b, v_b)
      &\lesssim \sum_e \left( \sum_{\k} \frac{p_\k^4}{h_\k^2} \| v_e \|_{0,\k}^2
      + \sum_{e'} \eta \frac{p_{e'}^2}{h_{e'}} \| \llb v_e \rrb \|_{0,e'}^2 \right)
      + \sum_{j} \sum_\k \eta \frac{p_\k^4}{h_\k^2} \| v_{j} \|_{0,\k}^2.
   \end{align}
   By Lemma \ref{lem:j-e-stable-decomp}, we have
   \begin{equation} \label{eq:jump-bounds}
      \sum_e \sum_{e'} \eta \frac{p_{e'}^2}{h_{e'}} \| \llb v_e \rrb \|_{0,e'}^2
      + \sum_{j} \sum_\k \eta \frac{p_\k^4}{h_\k^2} \| v_{j} \|_{0,\k}^2
      \lesssim \sum_{e'} \eta \frac{p_{e'}^2}{h_{e'}} \| \llb v_b \rrb \|_{0,e'}^2
   \end{equation}
   and by Corollary \eqref{cor:quadrature-2},
   \begin{equation} \label{eq:l2-bounds}
      \sum_e \sum_{\k} \frac{p_\k^4}{h_\k^2} \| v_e \|_{0,\k}^2
      \eq \sum_{\k} \frac{p_\k^4}{h_\k^2} \Big\| \sum_e v_e \Big\|_{0,\k}^2
      \lesssim \sum_{\k} \frac{p_\k^4}{h_\k^2} \| v_b \|_{0,\k}^2.
   \end{equation}
   Inserting \eqref{eq:jump-bounds} and \eqref{eq:l2-bounds} into \eqref{eq:ATB-inv-upper-bound}, noting that $(I-\Qh)v_b = v_b$ and applying Lemma \ref{lem:l2-jump-h1},
   \begin{align*}
      \A(T_B^{-1} v_b, v_b)
      &\lesssim \sum_\k \frac{p_\k^4}{h_\k^2} \| v_b \|_{0,\k}^2 + \sum_{e'} \eta \frac{p_{e'}^2}{h_{e'}} \| \llb v_b \rrb \|_{0,e'}^2 \\
      &\lesssim \sum_k \| \nabla v_b \|_{0,\k}^2 + \sum_{e'} \eta \frac{p_{e'}^2}{h_{e'}} \| \llb v_b \rrb \|_{0,e'}^2 \\
      &= \| v_b \|_{\DG}^2 \lesssim \A(v_b, v_b),
   \end{align*}
   proving the claim.
\end{proof}

Finally, by means of these results, we may prove that the resulting preconditioned system $T = T_B + T_C$ is uniformly well conditioned with respect to the mesh size, polynomial degree, and penalty parameter $\eta$.
\begin{thm} \label{thm:conditioning}
   The preconditioned system $T = T_B + T_C$ is uniformly well-conditioned, i.e.
   \[
      \kappa(T) = {\lmax(T)}/{\lmin(T)} \eq 1,
   \]
   independent of mesh size, polynomial degree, and penalty parameter $\eta$.
\end{thm}
\begin{proof}
   By Lemma \ref{lem:inf-identity}, we have
   \begin{align*}
      \A(T^{-1} u_h, u_h)
      &= \inf_{u_b + u_c = u_h} \left( \A(T_B^{-1} u_b, u_b) + \A(T_C^{-1} u_c, u_c) \right)\\
      &\leq \A(T_B^{-1} (I-\Qh) u_h, (I-\Qh)u_h) + \A(T_C^{-1} \Qh u_h, \Qh u_h) \\
      &\lesssim \A( u_h - \Qh u_h, u_h - \Qh u_h) + \A(\Qh u_h, \Qh u_h) \\
      &\lesssim \A(u_h, u_h),
   \end{align*}
   using Lemmas \ref{lem:stable-decomp} and \ref{lem:TB-inv-bounds}.
   This establishes the upper bound.

   For the lower bound, by Lemma \ref{lem:TB-inv-bounds} and the triangle inequality, we have
   \begin{align*}
      \A(T^{-1} u_h, u_h)
      &= \inf_{u_b + u_c = u_h} \left( \A(T_B^{-1} u_b, u_b) + \A(T_C^{-1} u_c, u_c) \right)\\
      &\gtrsim \inf_{u_b + u_c = u_h} \left( \A(u_b, u_b) + \A(u_c, u_c) \right)\\
      &\geq \A(u_h, u_h). \qedhere
   \end{align*}
\end{proof}

\subsection{Matrix-free preconditioners for the conforming problem}
\label{sec:lor}

In the above analysis, the approximate projection $T_C$ onto the conforming subspace $V_C = V_h \cap H^1(\Omega)$ corresponds to a preconditioned system for the standard $H^1$-conforming finite element problem.
As long as this preconditioned system is well conditioned, i.e. it satisfies
\[
   \A_C( T_C^{-1} u_c, u_c ) \eq \A_C(u_c, u_c) \quad\text{for all $u_c \in V_C$},
\]
where $A_C(\dotdot)$ is the standard $H^1$ bilinear form $\A_C(u,v) = \int_\Omega \nabla u\cdot\nabla v\,d\bm x$, then the result of Theorem \ref{thm:conditioning} holds and the proposed preconditioner is efficient.

However, constructing a good preconditioner for the general $hp$-conforming problem is challenging, because assembling the linear system matrix can be costly at high polynomial degrees.
Therefore, in this section we propose \textit{matrix-free} preconditioners for the conforming problem which are much more efficient at higher orders.
While such preconditioners have been previously considered in the case of $h$-refinement with a fixed $p$, the extension to the general $hp$-refinement case is new and is one of the contributions of this paper.

Our approach is based on a low-order refined methodology \cite{Orszag1980,Canuto2010}.
It is well known that a low-order ($p=1$) finite element discretization on a Gauss--Lobatto refined mesh is spectrally equivalent to the high-order conforming discretization \cite{Canuto1994}.
This equivalence is often also referred to as the finite element method--spectral element method (FEM-SEM) equivalence \cite{Canuto2006}.
Low-order refined preconditioners with parallel subspace corrections have been used in the context of discontinuous Galerkin discretizations on conforming meshes in \cite{Pazner2019a}.
However, in the case of nonconforming meshes or variable polynomial degrees, the low-order refined meshes do not match at coarse element interfaces, and the refined spaces corresponding to different polynomial degrees are not nested.
An illustration of one such mesh is shown in Figure \ref{fig:lor-example}.
For these reasons it is not immediately clear how to construct an equivalent low-order refined discretization in either the case of $p$-refinement or nonconforming meshes.

\begin{figure}
   \centering
   \includegraphics{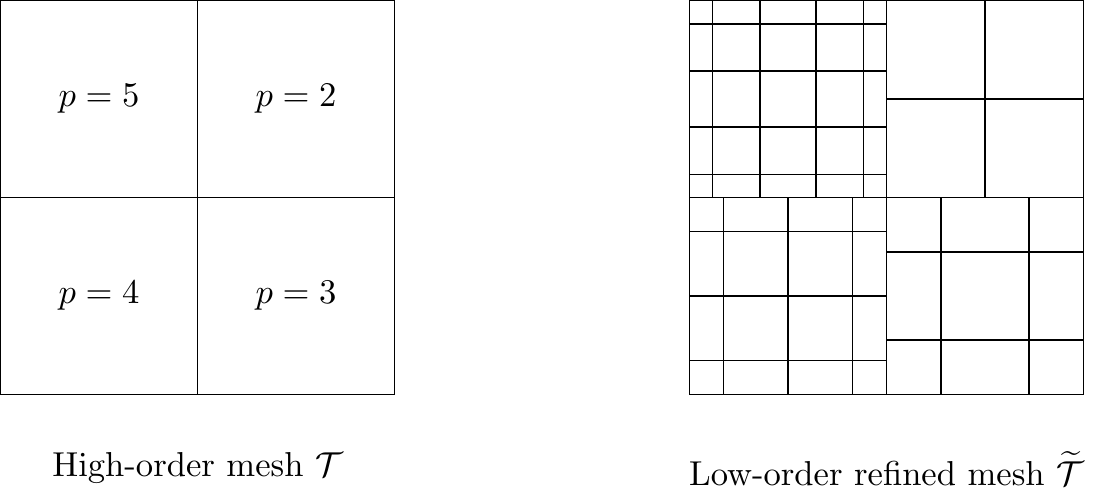}
   \caption{Example of high-order mesh $\T$ and corresponding low-order refined mesh $\widetilde{\T}$, illustrating the resulting non-matching, non-nested interfaces.}
   \label{fig:lor-example}
\end{figure}

In this work, we make use of the \textit{variational restriction} perspective for nonconforming adaptive mesh refinement \cite{Cerveny2019}.
Let $P : V_C \to V_h$ denote the natural injection, which we call the \textit{conforming prolongation} operator.
Let $\Ahat$ denote the matrix corresponding to the bilinear form $\widehat{\A}(u,v) = \int_\Omega \nabla_h u \cdot \nabla_h v \, d\bm x$.
Note that $\Ahat : V_h \to V_h$ is a block diagonal matrix, since it is defined on the ``broken'' DG space $V_h$.
Then, the matrix $A_C : V_C \to V_C$ corresponding to the conforming bilinear form $\A_C$ is given by the variational restriction
\[
   A_C = P^\tr \Ahat P.
\]

Let $\Ahat_\k$ denote the diagonal block of $\Ahat$ corresponding to element $\k$.
Let $\ktilde$ denote the low-order refined mesh of element $\k$, defined as the image under the element mapping $T_\k$ of the Cartesian mesh whose vertices are the $p_\k + 1$ tensor-product Gauss--Lobatto nodes.
Then, let $\Atilde_\k$ denote the matrix corresponding to the bilinear ($p=1$) finite element stiffness matrix on the low-order refined element mesh $\ktilde$.
The spectral equivalence of low-order refined discretizations (cf.\ \cite{Canuto1994}) implies that $\Ahat_\k$ is spectrally equivalent to $\Atilde_\k$, i.e.
\begin{equation} \label{eq:local-spectral-equivalence}
   u^\tr \Ahat_\k u \eq u^\tr \Atilde_\k u \qquad\text{for all $u$},
\end{equation}
independent of the polynomial degree $p_\k$.
Let $\Atilde$ be the block diagonal matrix whose diagonal blocks correspond to the low-order refined elemental matrices $\Atilde_\k$.
Then, define $\Atilde_C$ by
\[
   \Atilde_C = P^\tr \Atilde P.
\]
This definition gives us the following simple result.
\begin{prop} \label{prop:lor-conditioning}
   The low-order refined discretization $\Atilde_C$ is spectrally equivalent to the high-order conforming discretization $A_C$.
\end{prop}
\begin{proof}
   For any $u \in V_C$ we have, by \eqref{eq:local-spectral-equivalence} and setting $v = P u$,
   \[
      \frac{u^\tr A_C u}{u^\tr \Atilde_C u}
      = \frac{u^\tr P^\tr \Ahat P u}{u^\tr P^\tr \Atilde_C P u}
      = \frac{v^\tr \Ahat v}{v^\tr \Atilde_C v}
      \eq 1. \qedhere
   \]
\end{proof}

The advantage of the discretization $\Atilde_C$ is that the elemental matrices can be assembled in constant time per degree of freedom, as opposed to the high-order discretization $A_C$, for which naive implementations require $\mathcal{O}(p^{2d})$ operations per degree of freedom (optimized implementations using sum factorizations can reduce this cost to $\mathcal{O}(p^{d+1})$ operations per degree of freedom) \cite{Orszag1980,Melenk2001}.
Then, any uniform matrix-based preconditioner for $\Atilde_C$ can be used to precondition $\Atilde$.
In particular, algebraic multigrid methods such as BoomerAMG \cite{Henson2002}, which require an assembled matrix, can be applied easily.
In this work, the approximate projection $T_C$ is given by approximating $A_C^{-1}$ by one V-cycle of the BoomerAMG preconditioner applied to $\Atilde_C$.

\begin{rem}
   The low-order refined discretization $\Atilde_C$ is of interest in and of itself.
   Let $\widetilde{\T} = \{ \ktilde : \k \in \T \}$ denote the (non-matching) low-order defined mesh.
   In this context, the elements $\k\in\T$ will be referred to as macroelements, which will be refined to obtain the mesh $\widetilde{\T}$.
   Any high-order function $v_h \in V_h$ can be identified with a low-order refined function $\tilde{v}_h$, which is given on each element $\k$ by the low-order ($p=1$) interpolant at the Gauss--Lobatto nodes.
   Let $\widetilde{I} : v_h \mapsto \tilde{v}_h$ denote this identification, and let $\widetilde{V}_h = \widetilde{I}(V_h)$ denote the image of $\widetilde{I}$, which consists of piecewise $p=1$ functions defined on the refined mesh $\widetilde{\T}$ that are continuous within each macroelement $\k$, and potentially discontinuous across macroelements.
   We define the nonconforming finite element space $\widetilde{V}_C$ by
   \[
      \widetilde{V}_C = \left\{ \tilde{v}_h \in \widetilde{V}_h : \widetilde{I}^{-1}(\tilde{v}_h) \in H^1(\Omega)  \right\},
   \]
   that is, $\widetilde{V}_C$ consists of all low-order refined functions $\tilde{v}_h$, whose corresponding high-order function $v_h = \widetilde{I}^{-1}(\tilde{v}_h)$ is conforming.
   Then, the low-order refined $\Atilde_C$ can be seen to correspond to the bilinear form $\widetilde{\A} : \widetilde{V}_C \times \widetilde{V}_C \to \R$, $\widetilde{A}(\tilde{u}_h, \tilde{v}_h) = \int_\Omega \nabla_h \tilde{u_h} \cdot \nabla_h \tilde{v}_h \, d\bm x$.

   It is straightforward to see (using the norm-equivalence of low-order refined functions), that the operator $\widetilde{A}$ is bounded and coercive, with respect to the broken $H^1$ norm, denoted $\| \cdot \|_{1,h}$.
   Using techniques similar to those of mortar element methods \cite{Bernardi2005} together with a discrete Poincar\'e inequality, it is possible to bound the approximation and consistency errors of the discretization to obtain the error estimate
   \[
      \| \tilde{u}_h - u \|_{1,h} \lesssim h^{1/2} \| u \|_1.
   \]
   By the nonconforming Aubin-Nitsche lemma (cf.\ \cite{Braess2007}), we can obtain the $L^2$ estimate
   \[
      \| \tilde{u}_h - u \|_0 \lesssim h \| u \|_2.
   \]
   These estimates indicate that the discretization $\Atilde_C$ is of limited utility in terms of accuracy of the discrete solution, however, because of Proposition \ref{prop:lor-conditioning}, it will be quite useful for preconditioning the high-order problem.
\end{rem}

\section{Numerical results}
\label{sec:results}

\subsection{Implementation and algorithmic details}

The algorithms described in this paper have been implemented in the framework for the MFEM finite element library \cite{Anderson2020, MFEMWebsite}.
The main components of the solver are:
\begin{enumerate}
   \item Efficient matrix-free evaluation of the high-order discontinuous Galerkin bilinear form $\A(\dotdot)$.
   \item Assembly of the diagonal of the discontinuous Galerkin matrix (corresponding to the subspaces $V_j$), and assembly of the diagonal blocks corresponding to the subspaces $V_e$.
   \item Assembly of the low-order refined conforming stiffness matrix $\Atilde_C$.
   \item Assembly of the conforming prolongation operator $P$.
   \item Application of a uniform preconditioner (e.g.\ BoomerAMG) approximating $\Atilde_C^{-1}$.
\end{enumerate}

We now consider the number of operations required to perform these operations.
In particular, we are interested in the scaling with respect to polynomial degree.
For operations which are local to a given element or edge, the polynomial degree $p$ will be used to refer to $p_\k$ or $p_e$, respectively.
Using matrix-free sum-factorized operator evaluation, the evaluation of the high-order discontinuous Galerkin bilinear form $\A(\dotdot)$ can be performed element-by-element, requiring $\mathcal{O}(p^{d+1}) = \mathcal{O}(p^3)$ operations and constant memory per degree of freedom \cite{Pazner2018e,Orszag1980,Melenk2001}.
Additionally, the diagonal of the matrix can be assembled in the same complexity.
Since the spaces $V_e$ are composed of degrees of freedom lying on edges, the size of $V_e$ scales like $\mathcal{O}(n_e p_e)$, where $n_e$ is the number of edges that are included in $V_e$ through the generating process.
As a result, the assembly and inversion of these local blocks can be computed in $\mathcal{O}(p^3)$ operations, which is the same scaling as operator evaluation.
The low-order refined conforming stiffness matrix has $\mathcal{O}(1)$ nonzeros per row, and therefore can be assembled in constant time per degree of freedom (i.e.\ $\mathcal{O}(p^2)$ operations).
The number of nonzeros in the conforming prolongation operator scales as $\mathcal{O}(p^2)$, and each nonzero entry can be computed in constant time.
Finally, the construction and application of the BoomerAMG preconditioner also requires constant operations per degree of freedom.

In the numerical examples below, we will study the performance of the preconditioners developed in this paper applied to several problems involving nonconforming mesh refinement and variable polynomial degree.
We consider the preconditioned system $T = T_B + T_C$, where $T_B$ is defined by \eqref{eq:TB} and denote the corresponding preconditioner as $B = B_B + B_C$, such that $T = BA$.
In addition to $B$, we also consider a simplified preconditioner $\widetilde{B} = J_B + B_C$, where $J_B$ corresponds to a simple Jacobi method applied to the subspace $V_B$ (that is, the edge spaces $V_e$ are not used in the simplified preconditioner).
This preconditioner is known to be uniform for the case of conforming meshes with uniform polynomial degree \cite{Antonietti2016,Pazner2019a}, however the failure of estimates of the form \eqref{eq:Qh-jump} to hold in the $hp$-refinement case suggests that this preconditioner will not perform well in situations similar to that shown in Figure \ref{fig:mesh-example}.
Finally, we will compare our results to one V-cycle of the BoomerAMG algebraic multigrid method with Gauss-Seidel smoothing applied to the DG system.
Conjugate gradient iteration counts are reported using a relative tolerance of $10^{-8}$.
In the examples below, the penalty parameter $\eta$ is fixed to be $\eta = 100$, with the exception of Section \ref{sec:penalty}, in which we vary $\eta$ to study the impact of the penalty parameter on the preconditioner performance.

\subsection{Adaptive refinement}

We consider two examples of adaptively refined meshes and spaces.
The first problem is the standard L-shaped domain test \cite{Szabo1991,Solin2008},
\[
\begin{aligned}
   \Delta u &= 0 &&\quad\text{in $\Omega$,}\\
   u &= g_D &&\quad\text{on $\partial\Omega$,}
\end{aligned}
\]
where the Dirichlet boundary conditions are given by
\[
   g_D(\bm x) = r(\bm x)^{2/3} \sin\left( \frac{2\theta(\bm x)}{3} + \frac{\pi}{3} \right),
\]
where $r$ and $\theta$ denote polar coordinates in $\mathbb{R}^2$.
The exact solution to this problem has a singular gradient, triggering refinements near the corner of the domain.
The mesh, polynomial degrees, and solution after 14 adaptive refinement steps are shown in Figure \ref{fig:l-shape}.

\begin{figure}
   \centering
   \settoheight{\imageheight}{%
   \includegraphics[width=0.4\linewidth]{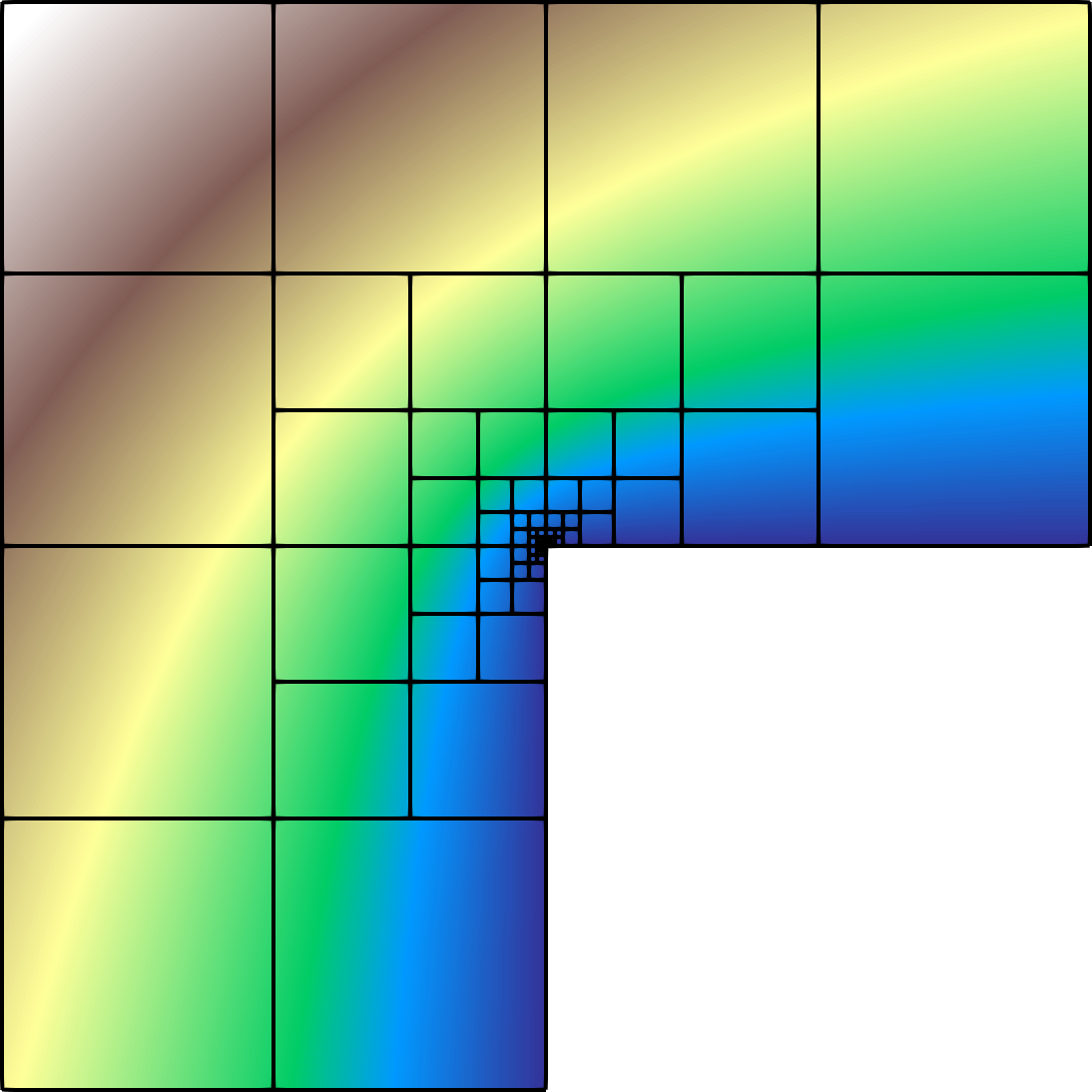}}
   \includegraphics[height=\imageheight]{fig/lshape/lshape} \hspace*{\fill}
   \includegraphics[height=\imageheight]{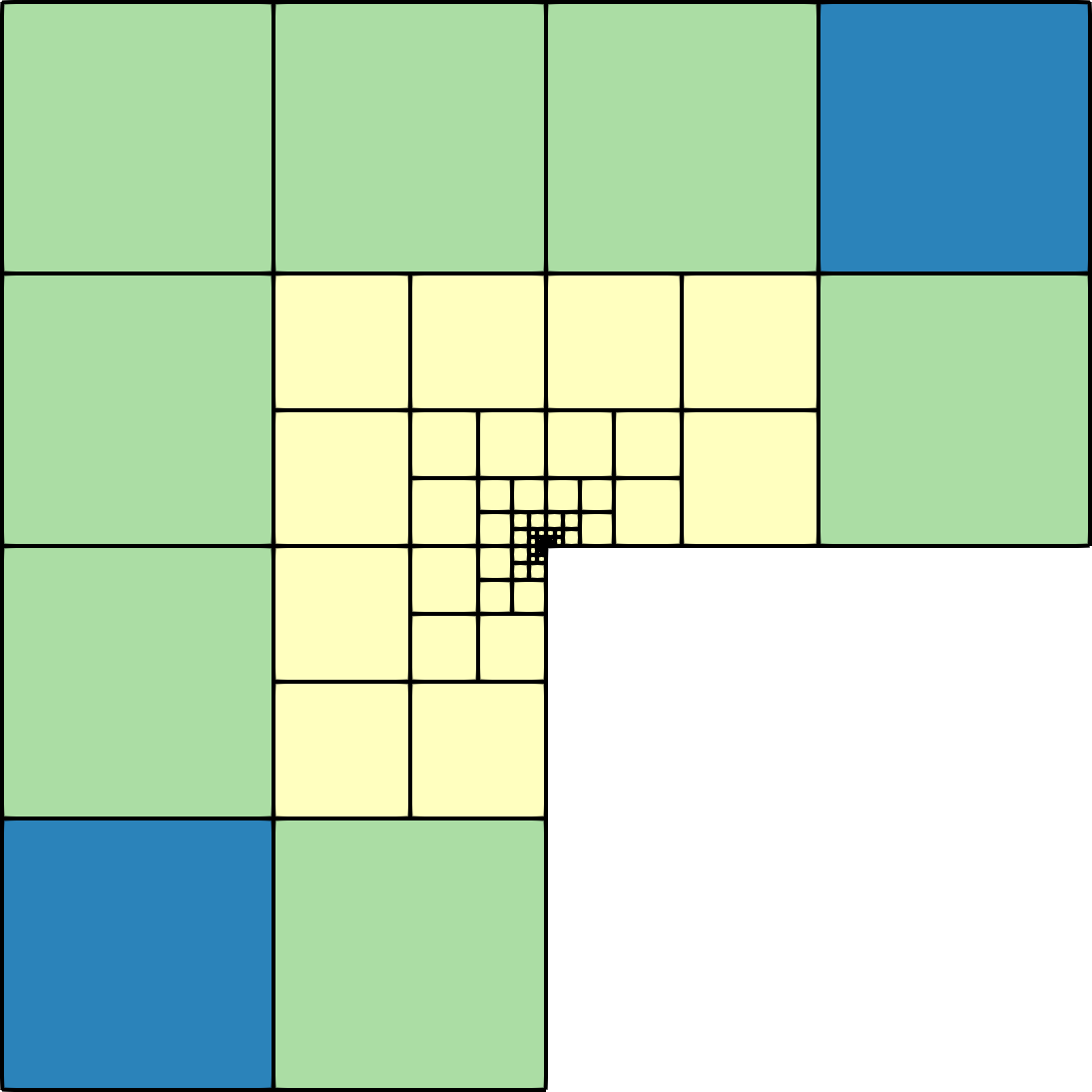}
   \includegraphics[height=\imageheight]{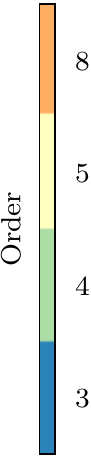}
   \caption{Solution and polynomial degrees for the L-shaped domain problem after 14 adaptive refinement steps.}
   \label{fig:l-shape}
\end{figure}

We adaptively refine the mesh and polynomial degrees 20 times, beginning with a coarse mesh with 12 elements.
The final mesh is 1-irregular, and consists of 138 elements, with a total of 5{,}207 degrees of freedom.
The conjugate gradient iteration counts for this problem are shown in Figure \ref{fig:amr-iters} and Table \ref{tab:lshape}.
We report the iteration counts for the subspace correction preconditioner $B = B_B + B_C$, where $T_B$ is defined by \eqref{eq:TB}, and $T_C$ is given by one V-cycle of BoomerAMG applied to the low-order refined discretization of the conforming problem.
The corresponding low-order refined mesh is illustrated in Figure \ref{fig:lor}.
Additionally, we show the iteration counts for the simplified preconditioner $\widetilde{B} = J_B + B_C$, where $J_B$ is a point Jacobi method applied to the subspace $V_B$.
This simplified preconditioner is expected to perform well on cases without problematic $hp$-interfaces of the type illustrated in Figure \ref{fig:mesh-example}.
Finally, we also consider one V-cycle of BoomerAMG with Gauss-Seidel smoothing applied to the DG problem.
We notice that for the L-shaped domain test case, because of the choice of $h$- and $p$-refinements, the first 18 refinements do not introduce problematic $hp$-interfaces, and both the subspace correction preconditioner $T$ and the simplified preconditioner $\widetilde{T}$ result in good performance that is essentially independent of refinement level.
The last two AMR steps introduce several problematic $hp$-interfaces, which result in degraded convergence for $\widetilde{T}$.
After only a couple of AMR steps, the BoomerAMG preconditioner applied to this problem results in large iteration counts.

\begin{table}
   \centering
   \caption{
      Convergence results and subspace sizes for the L-shaped domain test problem.
      $\# V_e$ indicates the number of nontrivial subspaces $V_e$ in the space decomposition \eqref{eq:VB-subspaces}.
      $\max\dim(V_e)$ indicates the dimension of the largest such subspace.
   }
   \label{tab:lshape}

   \vspace{\floatsep}

   \begin{tabular}{ccccc|cc}
      \toprule
      \makecell{AMR\\Step} & \# DOFs & \makecell{Iters.\\$B_B + B_C$} & \makecell{Iters.\\$J_B + B_C$} & \makecell{Iters.\\BoomerAMG} & \# $V_e$ & $\max \dim(V_e)$ \\
      \midrule
      0 & 63 & 24 & 24 & 25 & 0 & ---\\
      1 & 84 & 25 & 25 & 28 & 0 & ---\\
      2 & 121 & 27 & 27 & 36 & 0 & ---\\
      3 & 174 & 30 & 30 & 60 & 0 & ---\\
      4 & 223 & 30 & 30 & 72 & 0 & ---\\
      5 & 268 & 31 & 31 & 69 & 0 & ---\\
      6 & 319 & 32 & 32 & 87 & 0 & ---\\
      7 & 662 & 34 & 33 & 184 & 6 & 12\\
      8 & 1{,}014 & 32 & 32 & 207 & 6 & 12\\
      9 & 1{,}359 & 33 & 33 & 218 & 6 & 12\\
      10 & 1{,}683 & 33 & 33 & 226 & 6 & 12\\
      11 & 2{,}025 & 33 & 33 & 227 & 6 & 13\\
      12 & 2{,}367 & 33 & 33 & 234 & 6 & 13\\
      13 & 2{,}709 & 33 & 33 & 223 & 6 & 13\\
      14 & 3{,}042 & 32 & 32 & 223 & 6 & 13\\
      15 & 3{,}384 & 32 & 32 & 223 & 6 & 13\\
      16 & 3{,}752 & 35 & 35 & 209 & 2 & 13\\
      17 & 4{,}098 & 33 & 33 & 193 & 0 & ---\\
      18 & 4{,}422 & 34 & 34 & 194 & 0 & ---\\
      19 & 4{,}805 & 33 & 57 & 191 & 2 & 15\\
      20 & 5{,}207 & 35 & 71 & 198 & 8 & 25\\
      \bottomrule
   \end{tabular}
\end{table}

\begin{figure}
   \centering
   \includegraphics{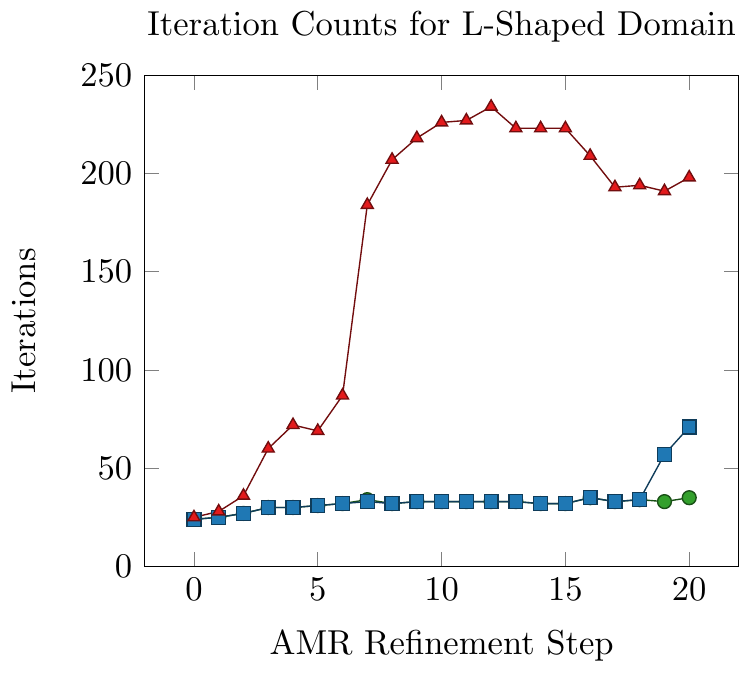}%
   \includegraphics{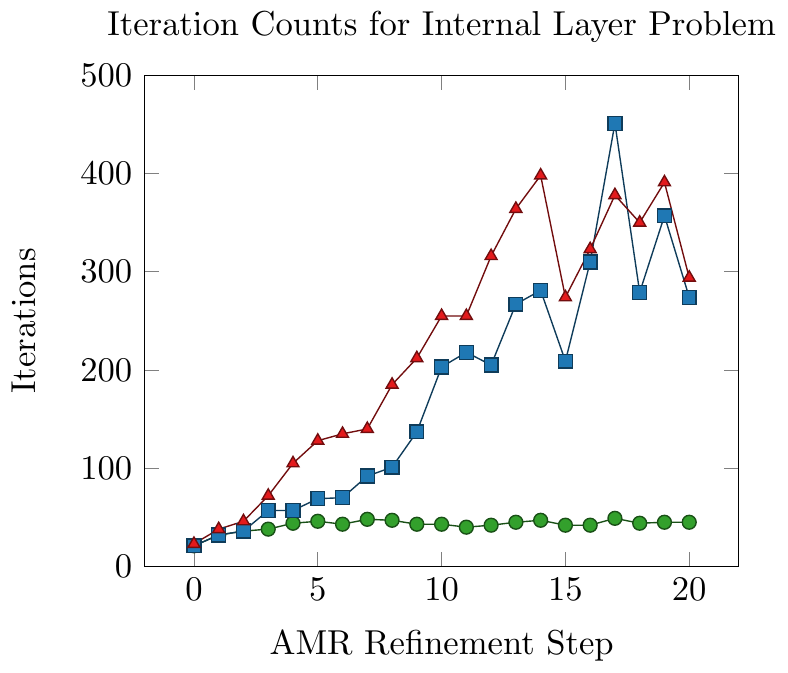}

   \vspace{\floatsep}
   \includegraphics{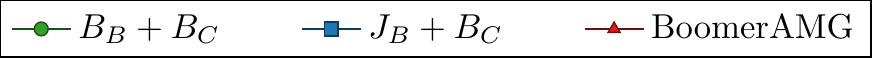}
   \caption{Conjugate gradient iteration counts for the L-shaped domain (left) and internal layer problem (right) adaptively refined problems.
   Comparison of subspace correction preconditioner $B = B_B + B_C$, simplified preconditioner $\widetilde{B} = J_B + T_C$, and BoomerAMG.}
   \label{fig:amr-iters}
\end{figure}

The second problem we consider is a problem with an internal layer \cite{Demkowicz2002,Solin2008}.
We solve the problem
\[
\begin{aligned}
   \Delta u &= f &&\quad\text{in $\Omega$,}\\
   u &= g_D &&\quad\text{on $\partial\Omega$,}
\end{aligned}
\]
where $f$ and $g_D$ are chosen to give the exact solution
\[
   u(\bm x) = \mathrm{atan}\left( 200 (r(\bm x) - 0.7) \right),
\]
where $r$ here denotes the distance from the point $(-0.05, -0.05)$.
This problem is characterized by a steep gradient near the circle centered at $(-0.05, -0.05)$ of radius $0.7$.
The results for this problem are shown in Figure \ref{fig:amr-iters} and Table \ref{tab:layer}.
In contrast to the previous test case, the adaptive refinement procedure in this case results in a large number of difficult $hp$-interfaces, as indicated by the large number of nontrivial $V_e$ subspaces.
As a consequence, the simplified preconditioner $\widetilde{T} = J_B = T_C$ does not perform well for this problem.
On the other hand, the subspace correction preconditioner results in iteration counts that remain bounded independent of the refinement level.
We note that the number of nontrivial subspaces $V_e$ increases with refinement level, thus increasing the cost of the preconditioner.
However, the majority of these subspaces are quite small, and they can be processed independently and in parallel.
For example, after 20 AMR steps, there are 121 nontrivial $V_e$ subspaces, of which the majority have dimension less than 10, and only 6 of which have dimension greater than 20.

\begin{figure}
   \centering
   \settoheight{\imageheight}{%
   \includegraphics[width=0.4\linewidth]{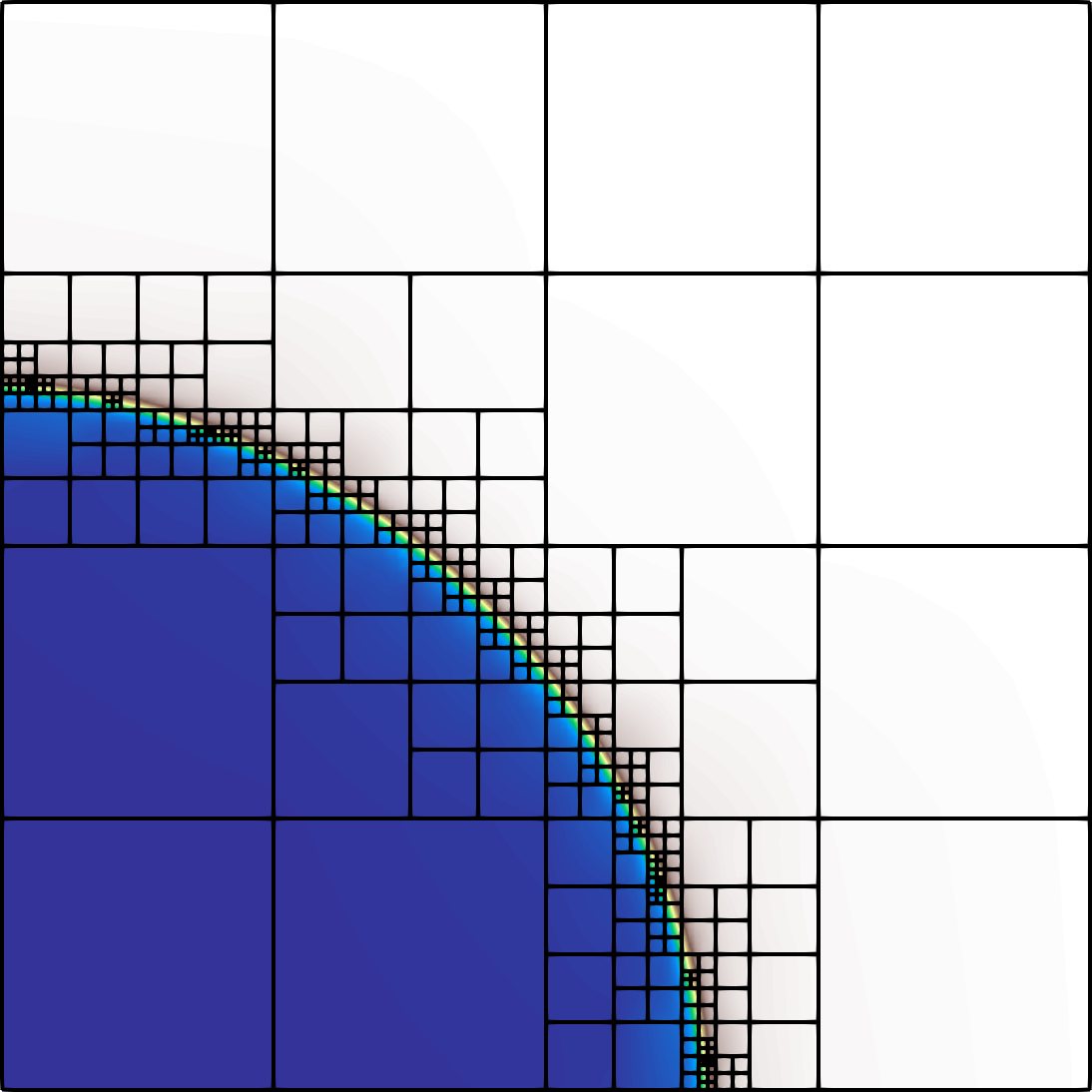}}
   \includegraphics[height=\imageheight]{fig/layer/layer} \hspace*{\fill}
   \includegraphics[height=\imageheight]{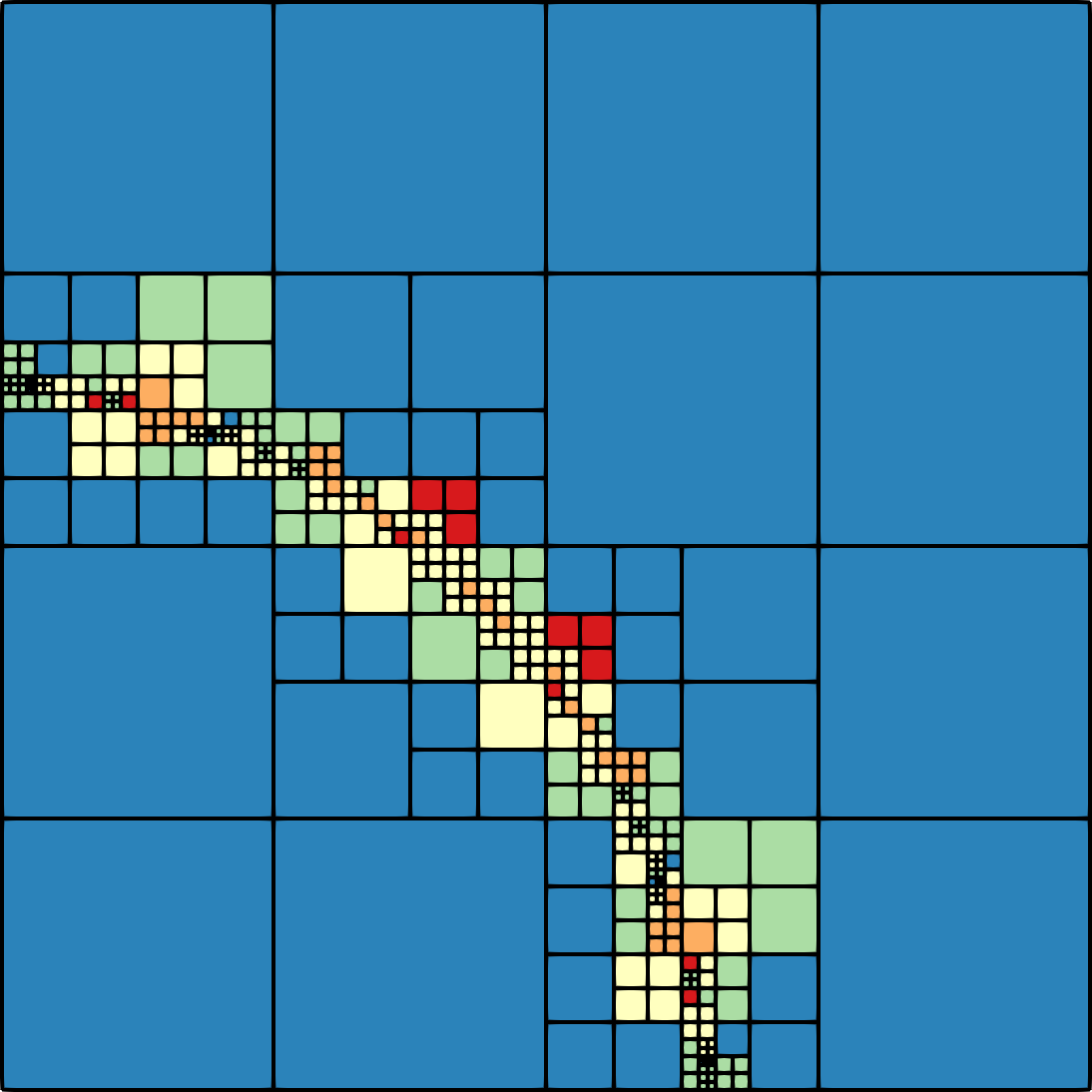}
   \includegraphics[height=\imageheight]{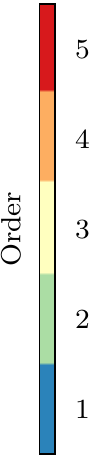}
   \caption{Solution and polynomial degrees for the internal layer problem after 20 adaptive refinement steps.}
   \label{fig:layer}
\end{figure}

\begin{table}
   \centering
   \caption{
      Convergence results and subspace sizes for the internal layer test problem.
      $\# V_e$ indicates the number of nontrivial subspaces $V_e$ in the space decomposition \eqref{eq:VB-subspaces}.
      $\max\dim(V_e)$ indicates the dimension of the largest such subspace.
   }
   \label{tab:layer}

   \vspace{\floatsep}

   \begin{tabular}{ccccc|cc}
      \toprule
      \makecell{AMR\\Step} & \makecell{\# DOFs} & \makecell{Iters.\\($B_B + B_C$)} & \makecell{Iters.\\($J_B + B_C$)} & \makecell{Iters.\\(BoomerAMG)} & \# $V_e$ & $\max \dim(V_e)$ \\
      \midrule
      0 & 64 & 21 & 21 & 23 & 0 & ---\\
      1 & 124 & 32 & 32 & 38 & 0 & ---\\
      2 & 196 & 36 & 36 & 46 & 0 & ---\\
      3 & 279 & 38 & 57 & 72 & 13 & 8\\
      4 & 354 & 44 & 57 & 105 & 13 & 12\\
      5 & 448 & 46 & 69 & 128 & 22 & 12\\
      6 & 502 & 43 & 70 & 135 & 22 & 12\\
      7 & 806 & 48 & 92 & 140 & 36 & 16\\
      8 & 958 & 47 & 101 & 185 & 39 & 20\\
      9 & 1{,}229 & 43 & 137 & 212 & 42 & 22\\
      10 & 1{,}602 & 43 & 203 & 255 & 53 & 24\\
      11 & 1{,}846 & 40 & 218 & 255 & 58 & 22\\
      12 & 2{,}086 & 42 & 205 & 316 & 63 & 22\\
      13 & 2{,}365 & 45 & 267 & 364 & 72 & 28\\
      14 & 2{,}421 & 47 & 281 & 398 & 72 & 31\\
      15 & 2{,}547 & 42 & 209 & 274 & 76 & 24\\
      16 & 2{,}892 & 42 & 310 & 323 & 84 & 27\\
      17 & 3{,}256 & 49 & 451 & 378 & 93 & 27\\
      18 & 3{,}805 & 44 & 279 & 350 & 111 & 27\\
      19 & 4{,}253 & 45 & 357 & 391 & 121 & 33\\
      20 & 4{,}401 & 45 & 274 & 294 & 121 & 34\\
      \bottomrule
   \end{tabular}
\end{table}

\begin{figure}
   \centering
   \hspace{0.05\linewidth}
   \includegraphics[width=0.4\linewidth]{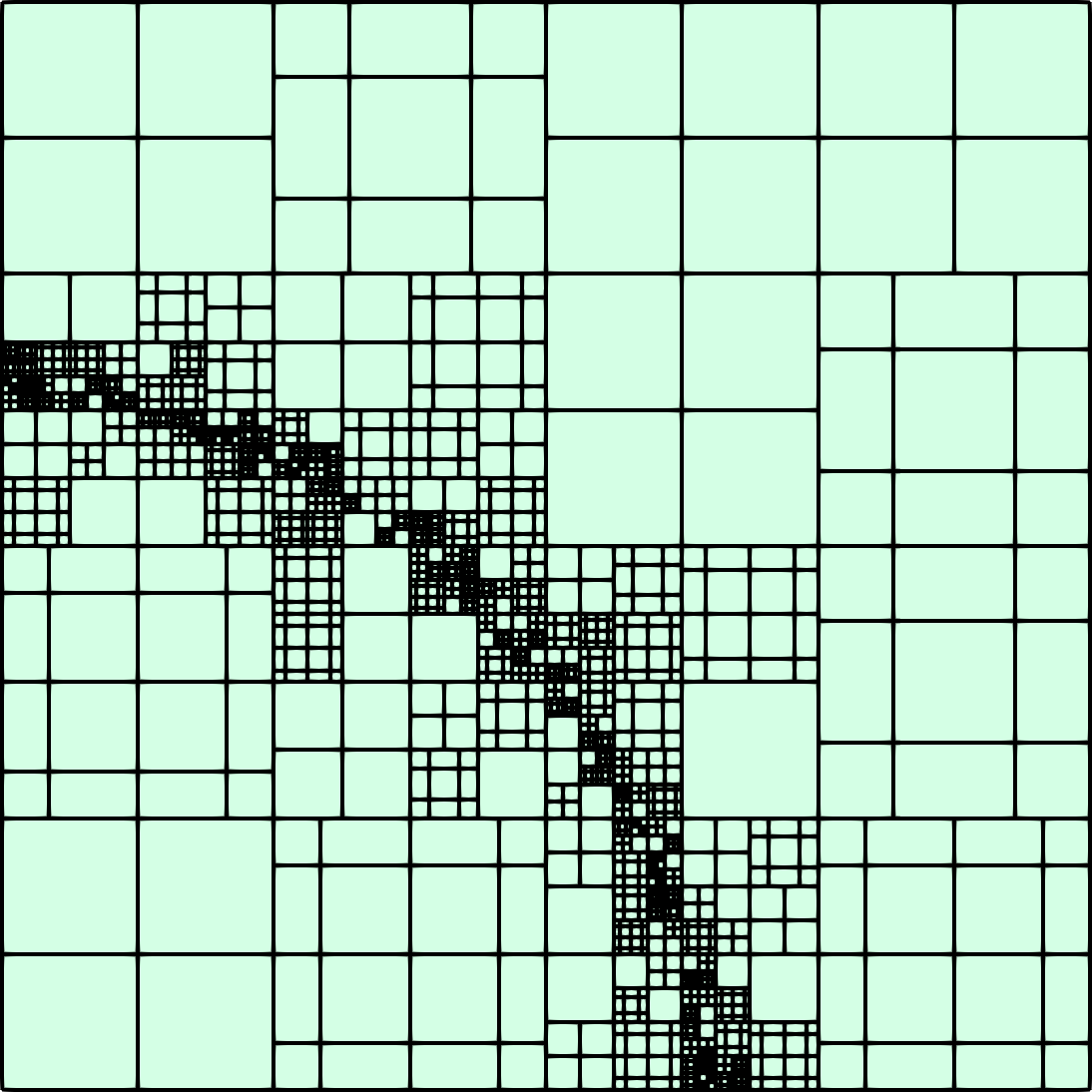}
   \hspace{\fill}
   \includegraphics[width=0.4\linewidth]{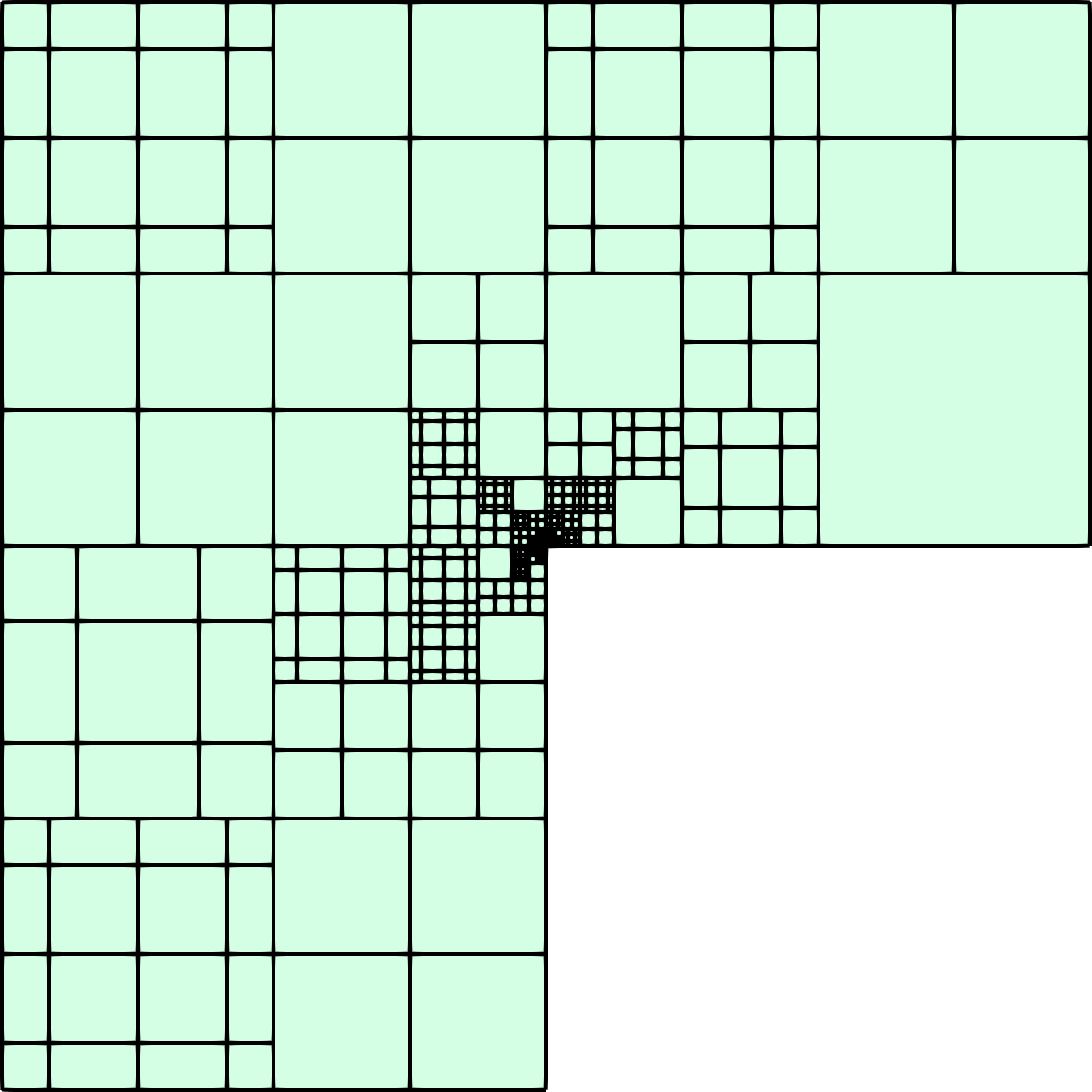}
   \hspace{0.05\linewidth}
   \caption{Low-order refined meshes for the L-shaped domain and internal layer problems.}
   \label{fig:lor}
\end{figure}

\subsection{Random refinement}
\label{sec:random-refinement}

To test the robustness of the new preconditioner, we now consider a sequence of nonconforming refinements made randomly.
Starting with an initial mesh, each element is marked for refinement with probability 0.5.
It is possible to limit the degree of irregularity of the final mesh (i.e. to ensure that an $\ell$-irregular mesh is obtained for given $\ell$) by propagating certain refinements.
We consider both the case of 1-irregular meshes and meshes with no limit on the degree of irregularity.
After the final mesh is obtained through this random refinement process, polynomial degrees are randomly assigned to each element.
We solve the problem
\[
   \nabla \cdot (a \nabla u) = f,
\]
where $a(\bm x)$ is a piecewise constant diffusion coefficient, which takes values of 1 and 20 according to a numbering of the elements the initial coarse mesh.
The coefficient and an example of a randomly refined mesh are shown in Figure \ref{fig:random}.

We study the convergence of the preconditioner for this problem using a combination of random and uniform refinements, and considering both 1-irregular meshes, and arbitrary $\ell$-irregular meshes.
The results are presented in Table \ref{tab:random}.
To begin, we refine the mesh once randomly, and then twice uniformly.
We note that the conjugate gradient iterations remain roughly constant with each uniform refinement.
Furthermore, the maximum dimension of the edge subspaces $V_e$ does not increase with uniform refinement.
We also consider increasing levels of random refinement.
We note that with increased random refinements (and increased irregularity of the mesh), we observe a slight degradation in the quality of the preconditioner.
Additionally, the dimension of the spaces $V_e$ is seen to grow rapidly with the irregularity of the mesh.
However, if we require that the mesh be 1-irregular, then the preconditioner performance and dimension of the spaces $V_e$ remains constant when performing random refinements.

\begin{figure}
   \centering
   \settoheight{\imageheight}{%
   \includegraphics[width=0.35\linewidth]{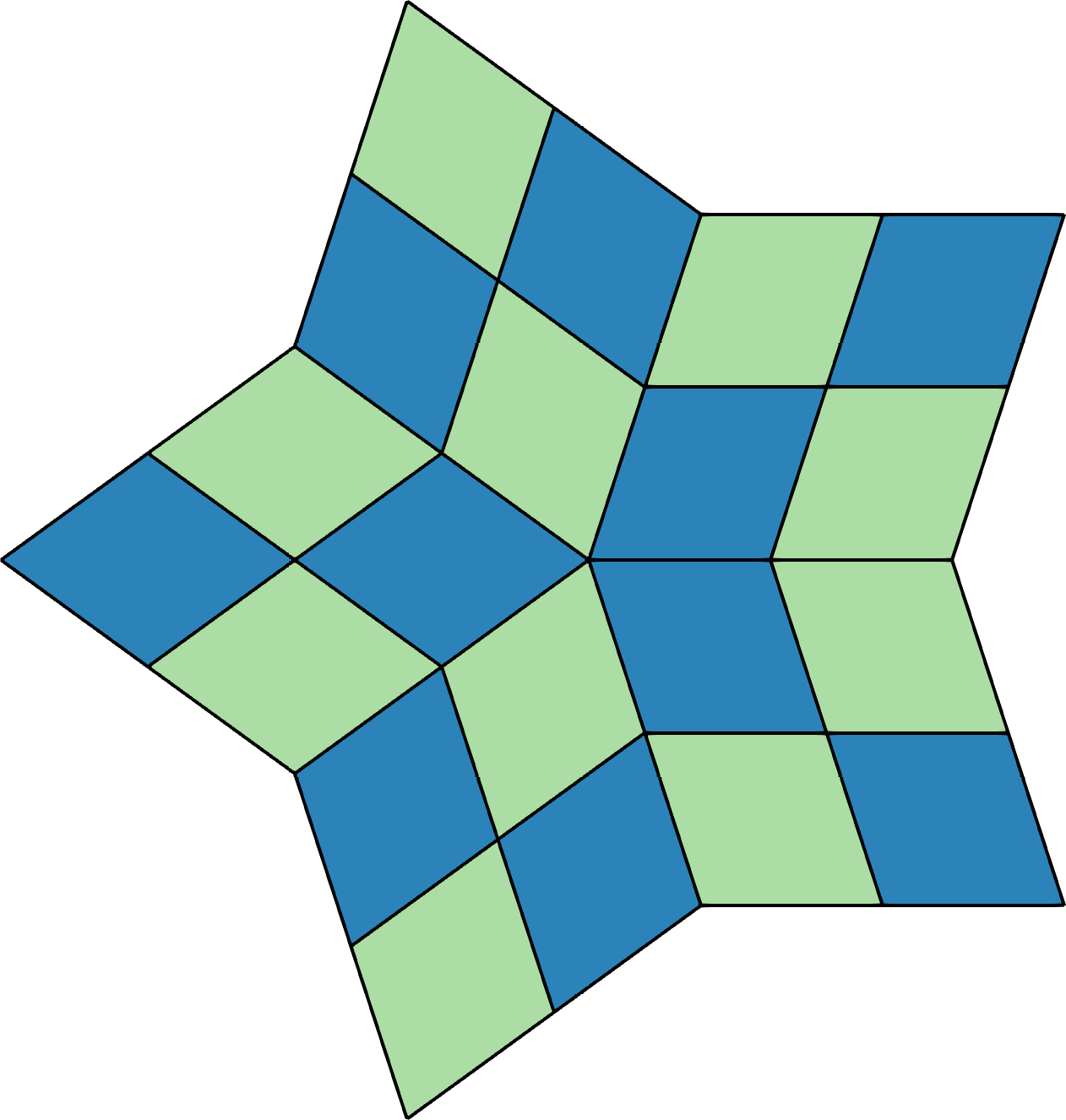}}
   \raisebox{-0.5\height}{\includegraphics[height=\imageheight]{fig/random/coeff}}
   \hspace{12pt}
   \raisebox{-0.5\height}{\includegraphics[height=0.85\imageheight]{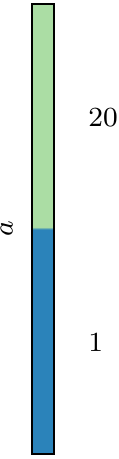}}
   \raisebox{-0.5\height}{\includegraphics[height=\imageheight]{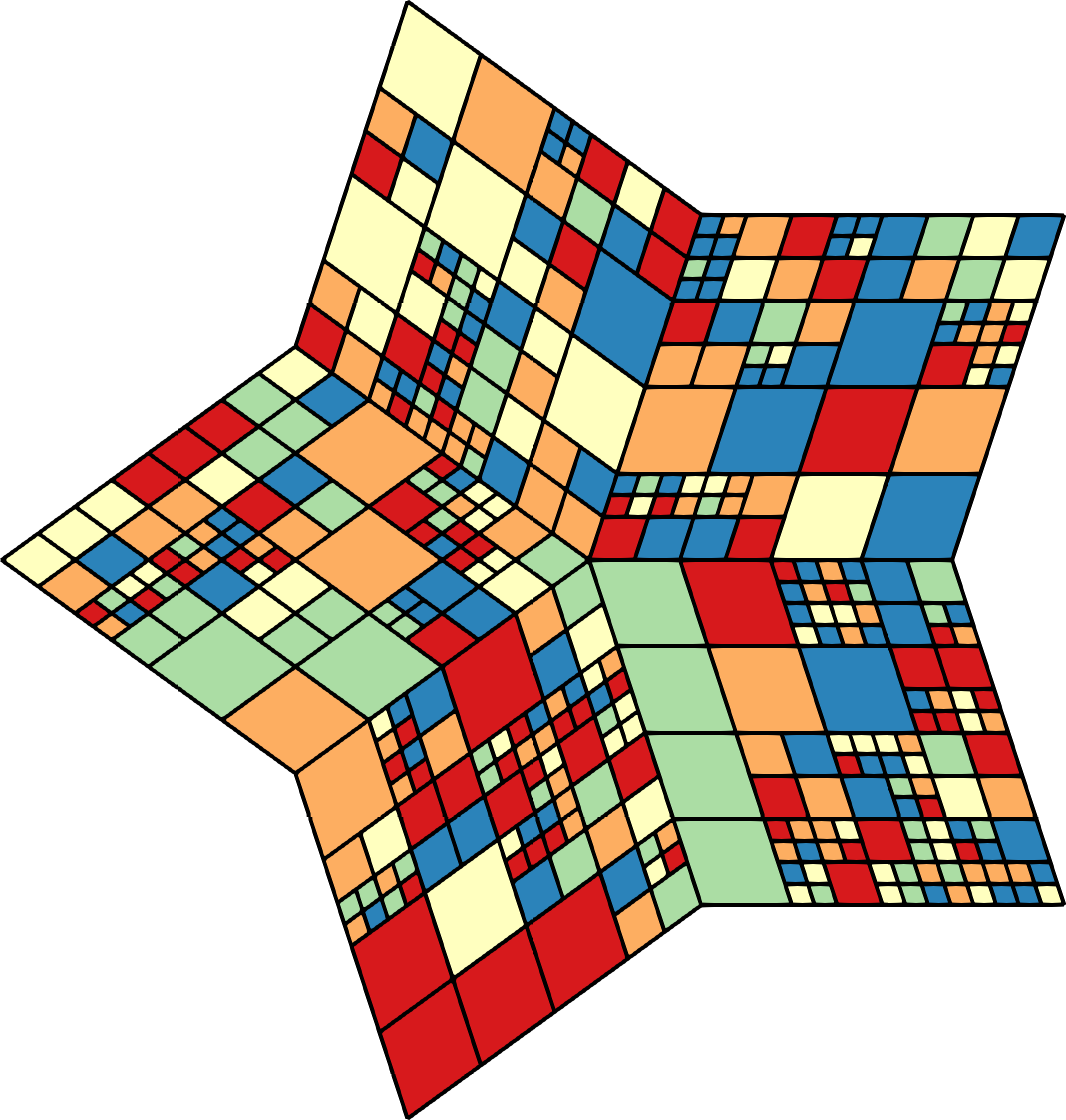}}
   \hspace{12pt}
   \raisebox{-0.5\height}{\includegraphics[height=0.85\imageheight]{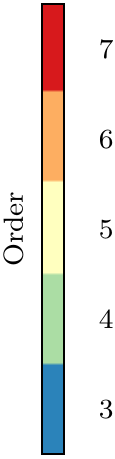}}

   \caption{Left: values of piecewise constant diffusion coefficient $a$.
   Right: example of randomly generated nonconforming mesh with randomly assigned polynomial degrees.}
   \label{fig:random}
\end{figure}

\begin{table}
   \centering
   \caption{Convergence results for the random refinement test case with piecewise constant diffusion coefficient.}
   \label{tab:random}
   \vspace{\floatsep}
   \begin{tabular}{lcccc|cc}
      \toprule
      Refinement & \# DOFs & \makecell{Iters.\\$T_B + T_C$} & \makecell{Iters.\\$J_B + T_C$} & \makecell{Iters.\\BoomerAMG} & \# $V_e$ & $\max \dim(V_e)$ \\
      \midrule
      {Initial mesh} & 672 & 42 & 42 & 148 & 0 & ---\\
      \midrule
      {1 random} & 2{,}042 & 54 & 223 & 237 & 18 & 29\\
      {1 random, 1 uniform} & 8{,}168 & 59 & 247 & 298 & 28 & 19\\
      {1 random, 2 uniform} & 32{,}672 & 61 & 316 & 288 & 56 & 19\\
      \midrule
      {2 random (2-irregular)} & 5{,}402 & 63 & 441 & 322 & 51 & 65\\
      {2 random (1-irregular)} & 5{,}737 & 60 & 305 & 288 & 54 & 29\\
      \midrule
      {3 random (3-irregular)} & 13{,}149 & 79 & 673 & 399 & 115 & 106\\
      {3 random (1-irregular)} & 15{,}300 & 61 & 390 & 286 & 138 & 30\\
      \bottomrule
   \end{tabular}
\end{table}

\subsection{Dependence on penalty parameter}
\label{sec:penalty}

An attractive feature of the preconditioners developed in this work is that the conditioning of the preconditioned system is independent of the value of the penalty parameter $\eta$.
Generally, larger values of the penalty parameter result in systems that are worse-conditioned, and more difficult to solve using standard preconditioners and multigrid methods \cite{Castillo2002,Shahbazi2005}.
In this section, we numerically study the dependence of the preconditioner on the choice of penalty parameter.
The same mesh is used as in Section \ref{sec:random-refinement}, with one level of random refinements.
Each element of the mesh is randomly assigned a polynomial degree $5 \leq p_\k \leq 9$.
The symmetric interior penalty parameter $\eta$ is increased from 10 to 10{,}000 by factors of 10.
The resulting iteration counts are shown in Figure \ref{fig:penalty}.
We note that the preconditioned system $T = T_B + T_C$ remains uniformly well-conditioned, independent of the choice of $\eta$, whereas both the simplified preconditioner $\widetilde{B}$ and BoomerAMG result in severely degraded convergence for large values of $\eta$.

We also consider an alternative DG formulation, known as the second method of Bassi and Rebay (BR2) \cite{Bassi2000}.
The BR2 method proceeds by defining, for each edge $e\in\Gamma$, a \textit{lifting operator} $r_e : [ L^1(e) ]^2 \to [ V_h ]^2$ given by
\begin{equation}
   \int_\Omega r_e(\bm \varphi) \cdot \bm \tau \, dx
      = - \int_e \bm \varphi \cdot \{ \bm \tau \} \, ds
      \qquad \text{for all $\bm\tau \in [V_h]^2$}.
\end{equation}
The BR2 bilinear form $\A_{\BR}$ is obtained by replacing the symmetric interior penalty term $\int_\Gamma \eta \frac{p_e^2}{h_e} \llb u_h \rrb \cdot \llb v_h \rrb \, ds$ with an alternative stabilization term of the form $\sum_e \int_\Omega \eta r_e(\llb u_h \rrb) \cdot r_e(\llb v_h \rrb) \, d\bm x$.
This method has the advantage that that the penalty parameter $\eta$ can be chosen to be $\mathcal{O}(1)$, and the scaling by a factor of $p_e^2/h_e$ is not required.
However, the presence of the lifting operators $r_e$ can result in degraded convergence for multigrid methods \cite{Fortunato2019}.
It can be seen that the norm $\|\cdot\|_{\BR}$ induced by the bilinear form $\A_{\BR}$ is equivalent to the DG norm $\|\cdot\|_{\DG}$, independent of mesh size and polynomial degree \cite{Pazner2019a}.
Consequently, we expect the preconditioner $B = B_B + B_C$ to result in performance independent of the magnitude of BR2 penalization.
This agrees with the numerical results presented in Figure \ref{fig:penalty}.

\begin{figure}
   \centering
   \includegraphics{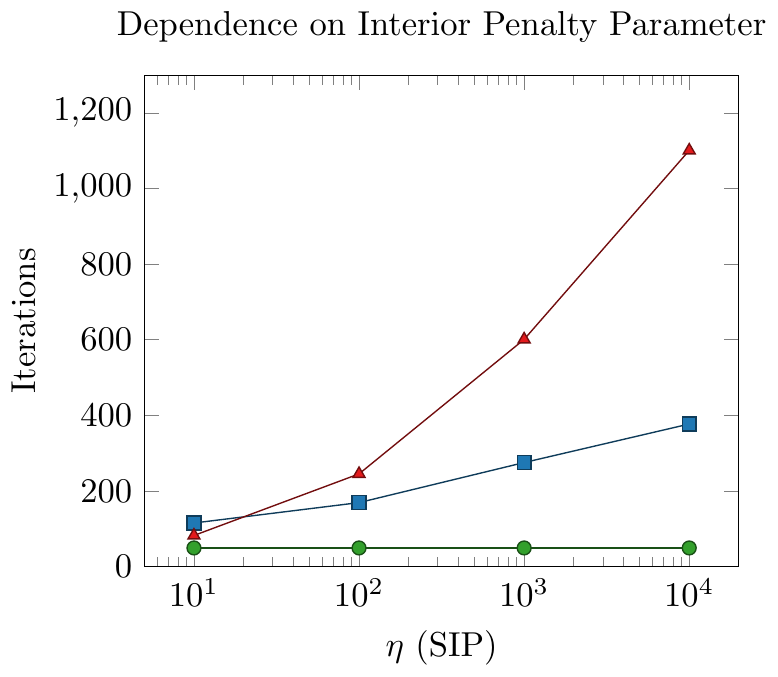}%
   \includegraphics{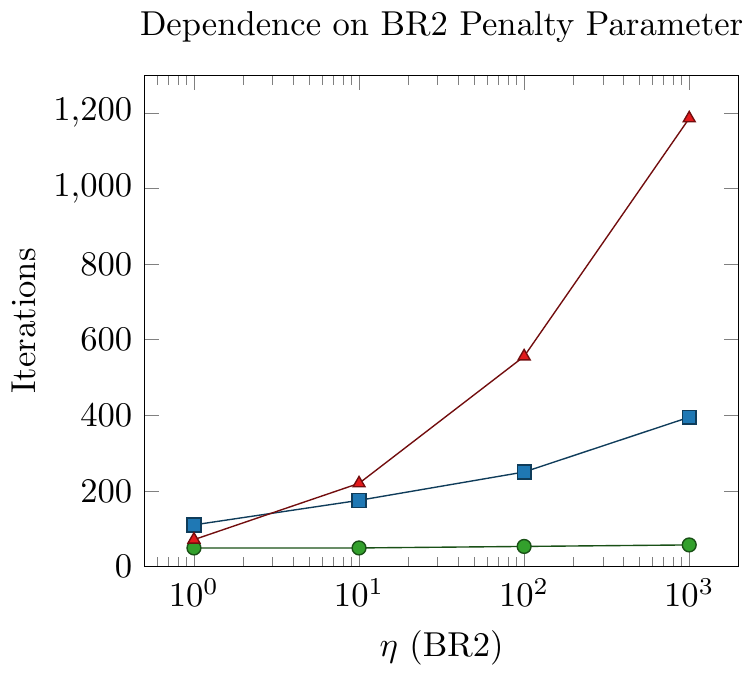}

   \vspace{\floatsep}
   \includegraphics{fig/amr_iters_legend}
   \caption{Dependence of number of conjugate gradient iterations on choice of penalty parameter.
   Left: $\int_\Gamma \eta \frac{p_e^2}{h_e} \llb u_h \rrb \cdot \llb v_h \rrb \, ds$ stabilization (symmetric interior penalty method).
   Right: $\sum_e \int_\Omega \eta r_e(\llb u_h \rrb) \cdot r_e(\llb v_h \rrb) \, d\bm x$ stabilization (BR2 method).}
   \label{fig:penalty}
\end{figure}

\section{Conclusions}
\label{sec:conclusions}

In this work we presented new preconditioners for discontinuous Galerkin methods applied to $hp$-refined meshes.
The preconditioners are based on a subspace decomposition, using a coarse space of $H^1$-conforming functions, together with subspaces corresponding to nonconforming interfaces.
For the coarse space we proposed a new matrix-free low-order refined preconditioner which is shown to be spectrally equivalent to the high-order conforming problem.
The nonconforming interface subspaces are generally small in size, and can be processed independently in parallel.
Analysis of the overall preconditioner shows that the condition number of the resulting system is independent of the mesh size, polynomial degree, and penalty parameter.
Numerical examples are presented on both adaptively refined and randomly refined meshes.
Comparisons to alternative preconditioners, including a simplified preconditioner with diagonal correction, and an algebraic multigrid preconditioner demonstrate the utility and benefits of the current approach.

\section{Acknowledgments}

This work was performed under the auspices of the U.S. Department of Energy by Lawrence Livermore National Laboratory under Contract DE-AC52-07NA27344 (LLNL-JRNL-814157).
This document was prepared as an account of work sponsored by an agency of the United States government.
Neither the United States government nor Lawrence Livermore National Security, LLC, nor any of their employees makes any warranty, expressed or implied, or assumes any legal liability or responsibility for the accuracy, completeness, or usefulness of any information, apparatus, product, or process disclosed, or represents that its use would not infringe privately owned rights.
Reference herein to any specific commercial product, process, or service by trade name, trademark, manufacturer, or otherwise does not necessarily constitute or imply its endorsement, recommendation, or favoring by the United States government or Lawrence Livermore National Security, LLC.
The views and opinions of authors expressed herein do not necessarily state or reflect those of the United States government or Lawrence Livermore National Security, LLC, and shall not be used for advertising or product endorsement purposes.

\bibliographystyle{siamplain}
\bibliography{refs}

\end{document}